\documentclass[10pt,reqno]{amsart}
\pdfoutput=1
\usepackage{adjustbox,array,bigints,cancel,color,comment,extpfeil,mathdots,mathrsfs,mathtools,MnSymbol,multicol,scalerel,setspace,tcolorbox,upgreek,hyperref}
\usepackage[fontsize = 10pt]{fontsize}
\usepackage[left = 2.5cm, right = 2.5cm, top = 2.5cm, bottom = 2.5cm]{geometry}
\usepackage{tikz}\usetikzlibrary{patterns}
\usepackage{tikz-cd}
\numberwithin{equation}{section}

\theoremstyle{plain}

\newtheorem{Cor}[equation]{Corollary}

\newtheorem{Lem}[equation]{Lemma}

\newtheorem{Prop}[equation]{Proposition}

\newtheorem{Thm}[equation]{Theorem}

\theoremstyle{definition}

\newtheorem{Defn}[equation]{Definition}
\newtheorem{Ex}[equation]{Example}

\newtheorem{Nt}[equation]{Notation}

\theoremstyle{remark}
\newtheorem{Rk}[equation]{Remark}
\newtheorem{Rks}[equation]{Remarks}

\theoremstyle{plain}
\newtheorem*{Cl*}{Claim}
\newtheorem*{Conj*}{Conjecture}
\newtheorem*{Lem*}{Lemma}
\newtheorem*{Prop*}{Proposition}
\newtheorem*{Q*}{Question}
\newtheorem*{Schol*}{Scholium}
\newtheorem*{SubCl*}{Subclaim}
\newtheorem*{Thm*}{Theorem}

\theoremstyle{definition}
\newtheorem*{Cond*}{Condition}
\newtheorem*{Cstr*}{Construction}
\newtheorem*{Defn*}{Definition}
\newtheorem*{Ex*}{Example}
\newtheorem*{Exs*}{Examples}
\newtheorem*{Md*}{Method}
\newtheorem*{Nt*}{Notation}
\newtheorem*{Pty*}{Property}

\theoremstyle{remark}

\newtheorem*{Rk*}{Remark}
\newtheorem*{Rks*}{Remarks}
\newtheorem*{A-d}{Aside}

\newcommand{\cag}{\begin{equation}\begin{gathered}}
\newcommand{\caag}{\end{gathered}\end{equation}}
\newcommand{\caw}{\begin{equation*}\begin{gathered}}
\newcommand{\caaw}{\end{gathered}\end{equation*}}
\newcommand{\e}{\begin{equation}\begin{aligned}}
\newcommand{\ee}{\end{aligned}\end{equation}}
\newcommand{\ew}{\begin{equation*}\begin{aligned}}
\newcommand{\eew}{\end{aligned}\end{equation*}}

\newcommand{\bcd}{\begin{tikzcd}}
\newcommand{\ecd}{\end{tikzcd}}
\newcommand{\bma}{\begin{matrix}}
\newcommand{\ema}{\end{matrix}}
\newcommand{\bpm}{\begin{pmatrix}}
\newcommand{\epm}{\end{pmatrix}}
\newcommand{\bvm}{\begin{vmatrix}}
\newcommand{\evm}{\end{vmatrix}}

\newcommand{\nts}{\begin{tcolorbox}}
\newcommand{\ntss}{\end{tcolorbox}}

\newcommand{\cref}[1]{Corollary \ref{#1}}

\newcommand{\dref}[1]{Definition \ref{#1}}

\newcommand{\eref}[1]{eqn.\hspace{0.6mm}(\ref{#1})}
\newcommand{\erefs}[1]{eqns.\hspace{0.6mm}(\ref{#1})}
\newcommand{\exref}[1]{Example \ref{#1}}

\newcommand{\lref}[1]{Lemma \ref{#1}}
\newcommand{\lrefs}[1]{Lemmas \ref{#1}}

\newcommand{\ntref}[1]{Notation \ref{#1}}
\newcommand{\pref}[1]{Proposition \ref{#1}}

\newcommand{\rref}[1]{Remark \ref{#1}}

\newcommand{\sref}[1]{\S\ref{#1}}

\newcommand{\tref}[1]{Theorem \ref{#1}}
\newcommand{\trefs}[1]{Theorems \ref{#1}}

\DeclareMathSizes{10}{10}{8}{7}

\newcommand{\mns}[1]{\mbox{\normalsize \(#1\)}}
\newcommand{\mla}[1]{\mbox{\large \(#1\)}}
\newcommand{\mLa}[1]{\mbox{\Large \(#1\)}}

\newcommand{\bb}[1]{\mathbb{#1}}
\newcommand{\cal}[1]{\mathscr{#1}}
\newcommand{\fr}[1]{\mathfrak{#1}}
\newcommand{\mb}[1]{\mbox{\boldmath \(#1\)}}
\newcommand{\mc}[1]{\mathcal{#1}}

\newcommand{\as}{\hspace{5mm}\text{as}\hspace{5mm}}
\newcommand{\et}{\hspace{5mm}\text{and}\hspace{5mm}}
\newcommand{\hs}[1]{\hspace{#1}}
\newcommand{\on}{\hspace{5mm}\text{on}\hspace{5mm}}
\newcommand{\vs}[1]{\vspace{#1}}

\DeclareMathSymbol{\Alpha}{\mathalpha}{operators}{"41}
\DeclareMathSymbol{\Beta}{\mathalpha}{operators}{"42}
\DeclareMathSymbol{\Epsilon}{\mathalpha}{operators}{"45}
\DeclareMathSymbol{\Zeta}{\mathalpha}{operators}{"5A}
\DeclareMathSymbol{\Eta}{\mathalpha}{operators}{"48}
\DeclareMathSymbol{\Iota}{\mathalpha}{operators}{"49}
\DeclareMathSymbol{\Kappa}{\mathalpha}{operators}{"4B}
\DeclareMathSymbol{\Mu}{\mathalpha}{operators}{"4D}
\DeclareMathSymbol{\Nu}{\mathalpha}{operators}{"4E}
\DeclareMathSymbol{\Omicron}{\mathalpha}{operators}{"4F}
\DeclareMathSymbol{\Rho}{\mathalpha}{operators}{"50}
\DeclareMathSymbol{\Tau}{\mathalpha}{operators}{"54}
\DeclareMathSymbol{\Chi}{\mathalpha}{operators}{"58}
\DeclareMathSymbol{\omicron}{\mathord}{letters}{"6F}

\newcommand{\al}{\alpha}
\newcommand{\be}{\beta}
\newcommand{\ga}{\gamma}

\newcommand{\ep}{\varepsilon}

\renewcommand{\th}{\theta}

\newcommand{\io}{\iota}
\newcommand{\ka}{\kappa}
\newcommand{\la}{\lambda}
\newcommand{\si}{\sigma}

\newcommand{\ph}{\phi}

\newcommand{\ch}{\chi}

\newcommand{\Ga}{\Gamma}

\newcommand{\Th}{\Theta}

\newcommand{\La}{\Lambda}
\newcommand{\Si}{\Sigma}

\newcommand{\Ph}{\Phi}

\newcommand{\Om}{\Omega}

\newcommand{\<}{\langle}
\newcommand{\?}{\rangle}

\newcommand{\ds}{\oplus}

\newcommand{\lqt}[2]{\left.\raisebox{-1mm}{\(#2\)}\middle\backslash\raisebox{1mm}{\(#1\)}\right.}

\newcommand{\rqt}[2]{\left.\raisebox{1mm}{\(#1\)}\middle/\raisebox{-1mm}{\(#2\)}\right.}
\newcommand{\ts}{\otimes}

\newcommand{\x}{\times}

\newcommand{\del}{\partial}

\newcommand{\Hocl}[1]{\overset{\circ}{H^{#1}}_{\kern-1.9mm\cl}}

\newcommand{\lop}{\left\|\kern-1.30mm\left\|}
\newcommand{\op}{\|\kern-1.30mm\|}
\newcommand{\rop}{\right\|\kern-1.30mm\right\|}
\newcommand{\SI}{\operatorname{\cal{I}\kern-1.5pt nd}}

\newcommand{\C}{\Subset}
\newcommand{\cc}{\subseteq}

\newcommand{\es}{\emptyset}

\newcommand{\mt}{\mapsto}

\newcommand{\osr}{\backslash}
\newcommand{\oto}[1]{\xrightarrow{#1}}
\newcommand{\pc}{\subset}

\newcommand{\B}{\mathrm{B}}

\newcommand{\cl}{\mathrm{closed}}

\newcommand{\dd}{\mathrm{d}}
\newcommand{\diam}{\operatorname{diam}}

\newcommand{\emb}{\hookrightarrow}

\newcommand{\M}{\mathrm{M}}

\newcommand{\s}{\odot}

\renewcommand{\ss}[2][{}]{\bigodot{\hspace{-1mm}}^{#2}_{#1}\hspace{0.6mm}}

\newcommand{\T}{\mathrm{T}}
\newcommand{\tl}{{\mns{\sim}}}

\newcommand{\w}{\wedge}

\newcommand{\0}{\infty}
\newcommand{\1}{\cdot}

\renewcommand{\ge}{\geqslant}
\newcommand{\gl}{\hspace{0.4mm}\raisebox{0.8mm}{\(>\)}\kern-1.8mm\raisebox{-0.8mm}{\(<\)}\hspace{0.4mm}}
\newcommand{\gle}{\hspace{0.4mm}\raisebox{1.2mm}{\(\ge\)}\kern-1.8mm\raisebox{-1.2mm}{\(\le\)}\hspace{0.4mm}}
\renewcommand{\le}{\leqslant}

\newcommand{\g}{\(\mathrm{G}_2\)}

\newcommand{\GL}{\operatorname{GL}}

\newcommand{\Stab}{\operatorname{Stab}}

\renewcommand{\iff}{if and only if}

\newcommand{\rmm}{Riemannian metric}
\newcommand{\rsm}{Riemannian semi-metric}
\newcommand{\sqfs}{stratified quasi-Finslerian structure}
\newcommand{\srm}{stratified Riemannian metric}
\newcommand{\srsm}{stratified Riemannian semi-metric}
\newcommand{\stp}{suffices to prove}
\newcommand{\Wlg}{Without loss of generality}
\newcommand{\wlg}{without loss of generality}

\newcommand{\wrt}{with respect to}

\newcommand{\GH}{Gromov--Hausdorff}

\newcommand{\ol}[1]{\overline{#1}}
\newcommand{\h}{\widehat}
\newcommand{\lt}{\left}
\newcommand{\m}{\middle}

\newcommand{\rt}{\right}

\newcommand{\tld}{\widetilde}

\title{A general collapsing result for families of stratified Riemannian metrics on orbifolds}
\author{Laurence H. Mayther}

\begin{document}\fontsize{10pt}{12pt}\selectfont
\begin{abstract}
\footnotesize{This paper proves a general collapsing result for families of stratified Riemannian metrics \(\h{g}^\mu\) on a compact orbifold \(E\), subject to suitable limiting conditions on the metrics \(\h{g}^\mu\) as \(\mu \to \0\).  The result is distinct from similar theorems in the literature since it does not require bounds on curvature or injectivity radius of \(\lt(E,\h{g}^\mu\rt)\) and thus allows for Gromov--Hausdorff limits of \(\lt(E,\h{g}^\mu\rt)\) which have strictly lower dimension than \(E\).  The paper also introduces and studies a new class of stratified fibrations between orbifolds, termed weak submersions, and new classes of geometric structures on orbifolds, termed \srm s, \srsm s and stratified quasi-Finslerian structures, all of which play a key role in the proof of the main theorem.}
\end{abstract}
\maketitle

\section{Introduction}

Let \(E_1\) be a compact orbifold and, for simplicity, let \(\lt(g^\mu\rt)_{\mu \in [1,\0)}\) be a family of \rmm s on \(E_1\).  Recall that each \(g^\mu\) induces a length structure on the set of piecewise-\(C^1\) paths in \(E_1\) (by integrating the length of the tangent vector along the path) and hence defines an intrinsic metric on \(E_1\), denoted \(d^\mu\) (see \sref{LS} for further details).  The main result of this paper, \tref{Cvgce-Thm-1-FV}, provides sufficient conditions for the metric spaces \(\lt(E_1,d^\mu\rt)\) to converge in the Gromov--Hausdorff sense to a compact orbifold \((B,d)\).  Additionally, \tref{Cvgce-Thm-1-FV} provides an explicit description of the length structure on piecewise-\(C^1\) paths in \(B\) inducing the metric \(d\).

The full statement \tref{Cvgce-Thm-1-FV} is somewhat technical.  Thus, for the purposes of the introduction, I state two special cases of the theorem, referring the reader to \sref{MT-section} for further details.  Let \(E_2\) be a second compact orbifold, let \(S_1\), \(S_2\) be compact subsets of \(E_1\), \(E_2\), respectively, and suppose that the complements \(E_1 \osr S_1\) and \(E_2 \osr S_2\) are diffeomorphic; henceforth, identify \(E_1 \osr S_1\) with \(E_2 \osr S_2\) and {\it vice versa}.  Let \(\lt(U^{(r)}_1\rt)_{r \in (0,1]}\) be open neighbourhoods of \(S_1\) in \(E_1\) such that \(U^{(r)}_1 \cc U^{(s)}_1\) for \(r < s\) and let \(\lt(U^{(r)}_2\rt)_{r \in (0,1]}\) denote the corresponding open neighbourhoods of \(S_2\) in \(E_2\).  Suppose that the regions \(U^{(r)}_i(j)\) are `compatibly fibred' in the sense that there exist surjective maps \(\fr{f}_i: U^{(1)}_i \to S\) to a common set \(S\) for \(i = 1,2\) such that the following diagram commutes:
\ew
\bcd
\lt. U^{(1)}_1(j) \m\osr S_1(j) \rt. \ar[rr, " \Ph "] \ar[rd, " \fr{f}_{1,j} "'] & & \lt. U^{(1)}_2(j) \m\osr S_2(j) \rt. \ar[dl, " \fr{f}_{2,j} "]\\
& S(j) &
\ecd
\eew

In this simplified case, the main theorem of this paper can be stated as follows:

\begin{Thm}[Simplified Version]\label{Cvgce-Thm-1}
Suppose that there exists a smooth map \(\pi: E_2 \to B\) which is a weak submersion (i.e.\ it can be written in local coordinates as the canonical projection \(\bb{R}^{\dim(E_2)} \to \bb{R}^{\dim(B)}\) modulo the action of orbifold groups; see \sref{submersions} for a full definition) and let \(g^\0\) be a \rsm\ on \(E_2\) which vanishes along the distribution \(\ker\dd\pi\) and is positive-definite transverse to \(\ker\dd\pi\).  Write \(d^\0\) for the semi-metric on \(E_2\) induced by \(g^\0\).  Suppose moreover that the following four conditions are satisfied:

(i) \ For all \(r \in (0,1]\), the \rmm s \(g^\mu\) converge uniformly to \(g^\0\) on the region \(E_1 \osr U^{(r)}_1\) and there exist constants \(\La_\mu(r) \ge 0\) such that:
\ew
\lim_{\mu \to \0} \La_\mu(r) = 1 \et \h{g}^\mu - \La_\mu(r)^2 \h{g}^\0 \text{ is non-negative definite on } E^{(r)}, \text{ for all } \mu \in [1,\0);
\eew

(ii)
\ew
\limsup_{r \to 0} ~ \limsup_{\mu \to \0} ~ \sup_{p \in S} ~ \diam_{d^\mu}\lt[ \fr{f}_1^{-1}(\{p\}) \cap U^{(r)}_1 \rt] = 0;
\eew

(iii)
\ew
\limsup_{r \to 0} ~ \sup_{p \in S} ~ \diam_{d^\0}\lt[ \fr{f}_2^{-1}(\{p\}) \cap U^{(r)}_2 \rt] = 0;
\eew

(iv)
\ew
\limsup_{r \to 0} ~ \limsup_{\mu \to \0} ~ \sup_{\lt.\overline{U^{(r)}_1} \m\osr U^{(r)}_1 \rt.} \lt| d^\mu - d^\0 \rt| = 0.
\eew

Then \(\lt(E_1,d^\mu\rt)\) converges to \(\lt(B,d\rt)\) in the \GH\ sense as \(\mu \to \0\) for some metric \(d\).  Moreover, if \(g^\0\) is the pullback of a \rmm\ \(h\) on \(B\) along \(\pi\), then \(d\) is simply the intrinsic metric induced by \(h\).  In general, the length structure on piecewise-\(C^1\) paths in \(\B\) which induces \(d\) is described by \eref{mcL}.  
\end{Thm}

As a further special case, one can consider \(S_1 = S_2 = \es\), yielding the following result, which is of independent interest:

\begin{Thm}[Second Simplified Version]\label{Cvgce-Thm-2}
Let \(E\) be a compact orbifold and suppose that there exists a weak submersion \(\pi: E \to B\).  Let \(g^\mu\) be a family of \rmm s on \(E\) and let \(g^\0\) be a \rsm\ on \(E\) which vanishes along the distribution \(\ker\dd\pi\) and is positive-definite transverse to \(\ker\dd\pi\).  Write \(d^\0\) for the semi-metric on \(E\) induced by \(g^\0\).  If the \rmm s \(g^\mu\) converge uniformly to \(g^\0\) and there exist constants \(\La_\mu \ge 0\) such that:
\ew
\lim_{\mu \to \0} \La_\mu = 1, \text{ and } \h{g}^\mu - \La_\mu^2 \h{g}^\0 \text{ is non-negative definite on } E, \text{ for all } \mu \in [1,\0),
\eew
then \(\lt(E,d^\mu\rt)\) converges to \(\lt(B,d\rt)\) in the \GH\ sense as \(\mu \to \0\) for some metric \(d\).  Moreover, if \(g^0\) is the pullback of a \rmm\ \(h\) on \(B\) along \(\pi\), then \(d\) is simply the intrinsic metric induced by \(h\).  In general, the length structure on piecewise-\(C^1\) paths in \(\B\) which induces \(d\) is described by \eref{mcL}.  
\end{Thm}

\trefs{Cvgce-Thm-1} and \ref{Cvgce-Thm-2} (and their generalisation \tref{Cvgce-Thm-1-FV}) differ from similar results in the literature (e.g.\ the main theorem in \cite{SCafSS}) in that they do not assume bounds on the curvature or the injectivity radius of \(\lt(E_1,g^\mu\rt)\); in particular, they allow for the possibility that the orbifolds \(\lt(E_1,g^\mu\rt)\) `collapse' as \(\mu \to \0\), meaning that \(\dim B < \dim E_2\).  This capability will be exploited in an upcoming paper by the author \cite{URftHFoG23F}, where these theorems will be used to prove that two closed 7-manifolds equipped with families of closed \g-structures converge to orbifolds of strictly lower dimension.

The notion of a weak submersion is, to the author's knowledge, new and has not previously been studied in the literature.  Informally, weak submersions are `stratified fibrations', which are locally trivial over points in the base space with isomorphic orbifold groups but which may have non-homotopic fibres over points with distinct orbifold groups.  Likewise, the notion of a \sqfs\ is new and has not previously appeared in the literature.  Informally, it consists of a family of quasinorms on the tangent cones to the orbifold \(B\) which vary continuously between points with isomorphic orbifold groups, but which can be (boundedly) discontinuous between points with different orbifold groups (see \sref{Strat-data} for a precise definition).  In a similar vein, one can allow \(g^\mu\) and \(g^\0\) to be \srm s and \srsm s, respectively, and it is this `fully stratified' result which is stated in \tref{Cvgce-Thm-1-FV} and which will be considered throughout the rest of this paper.  The possibility of allowing \(g^\mu\) and \(g^\0\) to be stratified is key for applications and will, in particular, be needed in the upcoming paper \cite{URftHFoG23F}.

The results of the present paper were obtained during the author's doctoral studies, which where supported by EPSRC Studentship 2261110.

\section{Preliminaries}

\subsection{Gromov--Hausdorff distance and forwards discrepancy}\label{H&GHD}

I begin by briefly reviewing \GH\ distance; see \cite[\S7.3]{ACiMG} for a more detailed discussion of this topic.  Let \((X,d)\) be a metric space. Given \(Y\cc X\), let \(\mc{N}_\eta(Y) = \lt\{x\in X~\middle|~d(x,Y) < \eta\rt\}\) be the open \(\eta\)-neighbourhood of \(Y\), where \(d(x,Y) = \inf\lt\{d(x,y)~\middle|~y \in Y\rt\}\).  Given \(A, B \cc X\) non-empty, closed and bounded, define the Hausdorff distance between \(A\) and \(B\) to be:
\ew
d_\mc{H}(A,B) = \inf\lt\{\eta > 0 ~\middle|~A \cc \mc{N}_\eta(B) \text{ and } B \cc \mc{N}_\eta(A)\rt\}.
\eew
Now let \((Y,d_Y)\) and \((Y',d_{Y'})\) be compact metric spaces. The Gromov--Hausdorff distance between \((Y,d_Y)\) and \((Y',d_{Y'})\) is defined to be:
\ew
d_\mc{GH}\lt[(Y,d_Y),(Y',d_{Y'})\rt] = \inf\bigg\{\eta>0~\bigg|~\parbox{8.5cm}{\begin{center}there exists \((X,d)\) together with isometric embeddings \(\io:Y\emb X, \io':Y'\emb X\) such that \(d_\mc{H}\lt[\io(Y),\io'(Y')\rt] \le \eta\)\end{center}}\bigg\}.
\eew

It can be shown that \(d_\mc{GH}\lt[(Y,d_Y),(Y',d_{Y'})\rt] = 0\) \iff\ \((Y,d_Y)\) and \((Y',\dd_{Y'})\) are isometric, and thus \(d_\mc{GH}\) defines a metric on the collection of isometry classes of compact metric spaces.  In light of this, one says that a family \((Y_i,d_i)_{i \in \bb{N}}\) of compact metric spaces converges to a compact metric space \((Y,d)\) in the Gromov--Hausdorff sense as \(i \to \0\) if \(d_\mc{GH}\lt[(Y_i,d_i),(Y,d)\rt] \to 0\).

\GH\ distance is closely related to the notion of \(\ep\)-isometry.  Recall the definition that for metric spaces \((X,d)\), \((X',d')\) and \(\ep>0\), an \(\ep\)-isometry is a set-theoretic function \(f:X \to X'\) (which need not be continuous) satisfying:
\begin{itemize}
\item For all \(x' \in X'\), there exists \(x \in X\) such that \(d'(f(x),x') \le \ep\) (\(f(X)\) is an `\(\ep\)-net' in \(X'\));
\item For all \(x,y \in X\): \(|d'(f(x),f(y)) - d(x,y)| \le \ep\).
\end{itemize}
The link between \GH\ distance and \(\ep\)-isometries can be quantified as follows:
\begin{Prop}\label{GH&FD}
Let \((X,d)\) and \((X',d')\) be compact metric spaces.   Then:
\ew
d_\mc{GH}\lt[(X,d),(X',d')\rt] \le 2\inf\lt\{\ep > 0 ~\m|~\text{There exists an \(\ep\)-isometry \(f:(X,d) \to (X',d')\)}\rt\}.
\eew
\end{Prop}
Motivated by this result, I make the following definition:
\begin{Defn}
Let \((X,d)\) and \((X',d')\) be compact metric spaces.   The forwards discrepancy between \((X,d)\) and \((X',d')\), denoted \(\fr{D}[(X,d)\to(X',d')]\), is defined to be:
\ew
\fr{D}[(X,d)\to(X',d')] = \inf\lt\{\ep > 0 ~\m|~\text{There exists an \(\ep\)-isometry \(f:(X,d) \to (X',d')\)}\rt\}.
\eew
(Note that the infimum is finite since, choosing \(\ep = \max\lt[\diam(X,d),(X',d')\rt]\), any map \(f:X \to X'\) is an \(\ep\)-isometry.)
\end{Defn}
The definition of forwards discrepancy naturally extends to semi-metric spaces.  Recall that a semi-metric \(d\) satisfies all the usual conditions of a metric, except that distinct points \(x\) and \(y\) are permitted to satisfy \(d(x,y) = 0\).  Given a semi-metric space \((X,d)\), define an equivalence relation \(\tl_d\) on \(X\) via \(x \tl_d y\) \iff\ \(d(x,y) = 0\).  Then \(d\) descends to define a metric \(d_{\tl}\) on the quotient \(\rqt{X}{\tl_d}\), and I term \(\lt(\rqt{X}{\tl_d},d_{\tl}\rt)\) the free metric space on \((X,d)\) (the name deriving from the fact that the assignment \((X,d) \mt \lt(\rqt{X}{\tl_d},d_{\tl}\rt)\) is left-adjoint to the natural inclusion functor of the category {\bf Met} of metric spaces and non-expansive maps into the category {\bf SMet} of semi-metric spaces and non-expansive maps).  I say that a semi-metric space is compact if its corresponding free metric space is compact in the usual sense.  It is clear that the notion of forwards discrepancy is well-defined not just on the class of compact metric spaces, but also on the class of compact semi-metric spaces.

Next, I prove the key result concerning \GH\ distance and forwards discrepancy which shall be required in this paper:
\begin{Prop}\label{2stage}
Let \((X,d)\) be a compact metric space and let \((X',d')\) be a compact semi-metric space.   Then:
\ew
d_\mc{GH}\lt[(X,d),\lt(\rqt{X'}{\tl_{d'}},d'_{\tl}\rt)\rt] \le 2\fr{D}[(X,d) \to (X',d')].
\eew
In particular, given a family of compact metric spaces \((X^\mu,d^\mu)_{\mu \in [1,\0)}\) such that \(2\fr{D}[(X^\mu,d^\mu) \to (X',d')] \to 0\) as \(\mu \to \0\), the spaces \((X^\mu,d^\mu)\) converge in the \GH\ sense to \(\lt(\rqt{X'}{\tl_{d'}},d'_{\tl}\rt)\) as \(\mu \to \0\).
\end{Prop}

\begin{proof}
Firstly note that whilst forwards discrepancy does not satisfy the triangle inequality, it does satisfy the weaker inequality:
\ew
\fr{D}[(X,d) \to (X'',d'')] \le \fr{D}[(X,d) \to (X',d')] + 2\fr{D}[(X',d') \to (X'',d'')]
\eew
for any compact semi-metric spaces \((X,d)\), \((X',d')\) and \((X'',d'')\): indeed, given \(\ep\)-isometries \(f:X \to X'\) and \(f':X'\to X''\) for some \(\ep,\ep'>0\), one may verify directly that \(f' \circ f:X \to X''\) is an \((\ep + 2\ep')\)-isometry.  Secondly, note that given a compact semi-metric space \((X,d)\), the quotient map \(f:(X,d) \to \lt(\rqt{X}{\tl_d},d_{\tl}\rt)\) is an \(\ep\)-isometry for any \(\ep>0\) and thus \(\fr{D}\lt[(X,d), \lt(\rqt{X}{\tl_d},d_{\tl}\rt)\rt] = 0\).  The result follows by combining these two observations with \pref{GH&FD}.

\end{proof}

\pref{2stage} decouples the task of computing \GH\ limits into two distinct stages: firstly, given a family \((X^\mu,d^\mu)_{\mu\in[1,\0)}\) of metric spaces, one finds a compact semi-metric space \((X^\0,d^\0)\) such that \(\fr{D}[(X^\mu,d^\mu) \to (X^\0,d^\0)] \to 0\) as \(\mu \to \0\).  By applying \pref{2stage}, it follows that \((X^\mu, d^\mu) \to \lt(\rqt{X^\0}{\tl_{d^\0}},d^\0_{\tl}\rt)\) in the \GH\ sense, as \(\mu \to \0\); thus, the task of computing the \GH\ limit of the family \((X^\mu,d^\mu)\) is reduced to describing the free metric space \(\lt(\rqt{X^\0}{\tl_{d^\0}},d^\0_{\tl}\rt)\).

\subsection{Length structures}\label{LS}

For a detailed treatment of length spaces, see \cite[Ch.\ 2]{ACiMG}.  Let \(X\) be a topological space.  A length structure on \(X\) is a class \(\mc{A}\) of continuous paths in \(X\) together with an assignment \(\ell: \mc{A} \to \bb{R} \cup \{\0\}\) satisfying the following four conditions:
\begin{enumerate}

\item If \(\lt(\ga: [a,b] \to X\rt) \in \mc{A}\) and \(c \in [a,b]\), then \(\lt(\ga|_{[a,c]}: [a,c] \to X\rt), \lt(\ga|_{[c,b]}: [c,b] \to X\rt) \in \mc{A}\) and:
\ew
\ell\lt(\ga\rt) = \ell\lt(\ga|_{[a,c]}\rt) + \ell\lt(\ga|_{[c,b]}\rt).
\eew
Moreover, \(\ell\lt(\ga|_{[a,c]}\rt)\) is continuous, when viewed as a function of \(c\).

\item If \(\ga: [a,b] \to X\) is continuous and \(c \in [a,b]\) is such that \(\lt(\ga|_{[a,c]}: [a,c] \to X\rt), \lt(\ga|_{[c,b]}: [c,b] \to X\rt) \in \mc{A}\), then \(\lt(\ga: [a,b] \to X\rt) \in \mc{A}\);

\item If \(\lt(\ga: [a,b] \to X\rt) \in \mc{A}\), and \(\ph: [c,d] \to [a,b]\) is a homeomorphism of the form \(t \mt \al t + \be\) (\(\al \ne 0\)), then \(\lt(\ga \circ \ph: [c,d] \to X\rt) \in \mc{A}\) and \(\ell\lt(\ga \circ \ph\rt) = \ell\lt(\ga\rt)\).

\item For all \(x \in X\) and all open neighbourhoods \(U\) of \(x\):
\ew
\inf \lt\{ \ell(\ga) ~\m|~ \lt(\ga: [a,b] \to X\rt) \in \mc{A} \text{ satisfies } \ga(a) = x \text{ and } \ga(b) \in X \osr U \rt\} > 0.
\eew
\end{enumerate}

Every length structure \((\mc{A},\ell)\) on \(X\) defines a metric \(d_{(\mc{A},\ell)}\) via:
\ew
d_{(\mc{A},\ell)}(x,y) = \inf\lt\{ \ell(\ga) ~\m|~ \lt(\ga: [a,b] \to X\rt) \in \mc{A} \text{ satisfies } \ga(a) = x \text{ and } \ga(b) = y \rt\}.
\eew
Such metrics are termed intrinsic.  In general, the topology on \(X\) induced by an intrinsic metric need only be no coarser than the original topology, in the sense that if \(U \pc X\) is open, then it is open \wrt\ \(d_{(\mc{A},\ell)}\) for any length-structure \((\mc{A},\ell)\).  However, for all the length-structures considered in this paper, the two topologies will in fact coincide.

Similarly, term \((\mc{A},\ell)\) a weak length structure if it satisfies conditions 1--3 above.  In this case, \((\mc{A},\ell)\) naturally induces a semi-metric \(d{(\mc{A},\ell)}\) on \(X\).

\subsection{Differential topology of orbifolds}

The material in this subsection is largely based on \cite[\S\S1.1--1.3]{O&ST} and \cite[\S14.1]{THKLFPFftScDO}.

\subsubsection{Basic definitions}

Let \(E\) be a topological space.  
\begin{Defn}
An \(n\)-dimensional orbifold chart \(\Xi\) is the data of a connected, open neighbourhood \(U\) in \(E\), a finite subgroup \(\Ga \pc \GL(n;\bb{R})\), a connected, \(\Ga\)-invariant open neighbourhood \(\tld{U}\) of \(0 \in \bb{R}^n\) and a homeomorphism \(\ch: \lqt{\tld{U}}{\Ga} \to U\).  Write \(\tld{\ch}\) for the composite \(\tld{U} \oto{quot} \lqt{\tld{U}}{\Ga} \oto{\ch} U\).  Say that \(\Xi\) is centred at \(e \in E\) if \(e = \tld{\ch}(0)\).   In this case, \(\Ga\) is called the orbifold group of \(e\), denoted \(\Ga_e\).  \(e\) is called a smooth point if \(\Ga_e = 0\), and a singular point if \(\Ga_e \ne 0\).

Now consider two orbifold charts \(\Xi_1 = \lt(U_1,\Ga_1,\tld{U}_1,\ch_1\rt)\) and \(\Xi_2 = \lt(U_2,\Ga_2,\tld{U}_2,\ch_2\rt)\) with \(U_1 \cc U_2\).   An embedding of \(\Xi_1\) into \(\Xi_2\) is the data of a smooth, open embedding \(\io_{12}:\tld{U}_1 \emb \tld{U}_2\) and a group isomorphism \(\la_{12}: \Ga_1 \to \Stab_{\Ga_2}(\io_{12}(0))\) such that for all \(x \in \tld{U}_1\) and all \(\si \in \Ga_1\): \(\io_{12}(\si \cdot x) = \la_{12}(\si) \cdot \io_{12}(x)\), and such that the following diagram commutes:
\ew
\bcd[column sep = 13mm, row sep = 3mm]
\tld{U}_1 \ar[rr, " \mla{\io_{12}} "] \ar[dd, " \mla{\tld{\ch}_1} "] & & \tld{U}_2 \ar[dd, " \mla{\tld{\ch}_2} "]\\
& & \\
U_1 \ar[rr, "\mla{incl}", hook] &  & U_2
\ecd
\eew

Now let \(\Xi_1\) and \(\Xi_2\) be arbitrary.  \(\Xi_1\) and \(\Xi_2\) are compatible if for every \(e \in U_1 \cap U_2\), there exists a chart \(\Xi_e = \lt(U_e,\Ga_e,\tld{U}_e,\ch_e\rt)\) centred at \(e\) together with embeddings \((\io_{e1},\la_{e1}): \Xi_e \emb \Xi_1\) and \((\io_{e2},\la_{e2}): \Xi_e \emb \Xi_2\).  If \(U_1 \cap U_2 = \es\), then \(\Xi_1\) and \(\Xi_2\) are automatically compatible, however if \(U_1 \cap U_2 \ne \es\) and \(\Xi_1\) and \(\Xi_2\) are compatible, then \(\Xi_1\) and \(\Xi_2\) have the same dimension; moreover, if \(\Xi_1\) and \(\Xi_2\) are centred at the same point \(e \in E\), then \(\Ga_1 \cong \Ga_2\) and therefore the orbifold group \(\Ga_e\) is well-defined up to isomorphism.

An orbifold atlas for \(E\) is a collection of compatible orbifold charts \(\fr{A}\) which is maximal in the sense that if a chart \(\Xi\) is compatible with every chart in \(\fr{A}\), then \(\Xi \in \fr{A}\).  An orbifold is a connected, Hausdorff, second-countable topological space \(E\) equipped with an orbifold atlas \(\fr{A}\).  Every chart of \(E\) has the same dimension \(n\); term this common dimension the dimension of the orbifold.
\end{Defn}

\begin{Defn}\label{c-inf-map}
Let \(E_1\), \(E_2\) be orbifolds.  A continuous map \(f:E_1 \to E_2\) is termed smooth if for any point \(e \in E_1\), there exists a chart \(\Xi_e = \lt(U_e,\Ga_e,\tld{U}_e,\ch_e\rt)\) for \(E_1\) centred at \(e\), a chart \(\Xi_{f(e)} = \lt(U_{f(e)},\Ga_{f(e)},\tld{U}_{f(e)},\ch_{f(e)}\rt)\) for \(E_2\) centred at \(f(e)\), a group homomorphism \(\ka_f:\Ga_e \to \Ga_{f(e)}\) and a smooth map \(\tld{f}: \tld{U}_e \to \tld{U}_{f(e)}\) satisfying \(\tld{f}(\si \cdot x) = \ka_f(\si) \cdot \tld{f}(x)\) for all \(x \in \tld{U}_e\) and \(\si \in \Ga_e\), such that the following diagram commutes:
\ew
\bcd[column sep = 13mm, row sep = 3mm]
\tld{U}_e \ar[rr, " \mla{\tld{f}} "] \ar[dd, " \mla{\tld{\ch}_e} "] & & \tld{U}_{f(e)} \ar[dd, " \mla{\tld{\ch}_{f(e)}} "]\\
& &\\
U_e \ar[rr, " \mla{f} "] &  & U_{f(e)}
\ecd
\eew
\end{Defn}

The lift \(\tld{f}\) need not be unique, even modulo the action of the groups \(\Ga_e\) and \(\Ga_{f(e)}\); see, e.g. \cite[Example 1.4.3]{ItO}.  Nevertheless, \dref{c-inf-map} is independent of the choice of charts \(\Xi_e\) and \(\Xi_{f(e)}\) and the map \(f\) has a well-defined differential in the following sense: the bottom arrow in the diagram:
\ew
\bcd[column sep = 10mm, row sep = 3mm]
\bb{R}^{n_1} \ar[rr, " \mla{D\tld{f}|_0} "] \ar[dd, " \mla{proj} "] & & \bb{R}^{n_2} \ar[dd, " \mla{proj} "]\\
& &\\
\lqt{\bb{R}^{n_1}}{\Ga_e} \ar[rr] & & \lqt{\bb{R}^{n_2}}{\Ga_{f(e)}}
\ecd
\eew
is independent of the choice of \(\tld{f}\).

\subsubsection{Suborbifolds and stratifications}

\begin{Defn}[\protect{See \cite[Defn.\ 13.2.7]{G&To3M}}]\label{suborb}
Let \(E\) be an orbifold.   A subset \(S \cc E\) is termed a suborbifold if for each \(e \in S\), there exists a chart \(\Xi_e = \lt(U_e,\Ga_e,\tld{U}_e,\ch_e\rt)\) for \(E\) centred at \(e\) and a \(\Ga\)-invariant subspace \(\bb{I}_e \pc \bb{R}^n\) such that:
\ew
\tld{\ch}_e{}^{-1}(S \cap U_e) = \tld{U}_e \cap \bb{I}_e.
\eew
Call such a chart regular for \(S\) and call \(\bb{I}_e\) the regular subspace.  If the action of \(\Ga\) on \(\bb{I}_e\) is trivial for all \(e \in S\), then call \(S\) a submanifold.
\end{Defn}

A subset \(S\) can have at most one suborbifold structure, as the following (readily verified) proposition demonstrates:
\begin{Prop}
Let \(E\) be an orbifold, \(S \cc E\) be a subset, \(e,f \in S\) and let \(\Xi_e = (U_e,\Ga_e,\tld{U}_e,\ch_e)\), \(\Xi_f = (U_f,\Ga_f,\tld{U}_f,\ch_f)\) be regular charts for \(S\) centred at \(e\) and \(f\) respectively.  Then \(\Xi_e\) and \(\Xi_f\) are compatible via regular charts.  Specifically, let \(g \in U_e \cap U_f \cap S\).  Then there exists a regular chart \(\Xi_g\) centred at \(g\) together with embeddings \((\io_{ge},\la_{ge})\) and \((\io_{gf},\la_{gf})\) into \(\Xi_e\) and \(\Xi_f\) respectively.  In particular, suborbifolds inherit a natural orbifold structure.
\end{Prop}

Using this terminology, one can make the following generalisation of Mather's terminology \cite[\S5]{NoTS} to orbifolds:

\begin{Defn}
Let \(E\) be an orbifold.   A stratification \(\Si\) of \(E\) is a partition of \(E\) into disjoint submanifolds \(E = \bigcup_{i=0}^n E_i\).  Say that \(\Si\) satisfies the condition of the frontier if, in addition, for each \(i \in \{0,...,n\}\), there exists \(I(i) \cc \{0,...,n\}\) such that:
\ew
\overline{E_i} = \bigcup_{j \in I(i)} E_j,
\eew
where \(\overline{E_i}\) denotes the topological closure of \(E_i\) in \(E\).

Now let \(\Si = \lt\{E_i\rt\}\), \(\Si' = \lt\{E'_j\rt\}\) be two stratifications of \(E\).   Say that \(\Si'\) is a refinement of \(\Si\) if for every \(j\), there exists an \(i\) such that \(E'_j \pc E_i\).  Finally, given stratified orbifolds \(\lt(E_1,\Si_1 = \lt\{E_{1,i}\rt\}_{i=1}^n\rt)\) and \(\lt(E_2,\Si_2= \lt\{E_{2,i}\rt\}_{i=1}^n)\rt)\), a smooth map \(f:E_1 \to E_2\) is a stratified diffeomorphism if it is an orbifold diffeomorphism (i.e.\ it has a smooth inverse) and \(f(E_{1,i}) = E_{2,i}\) for each \(i\) (in particular, \(\Si_1\) and \(\Si_2\) have the same number of strata).
\end{Defn}

\begin{Rk}
A stratification as defined above is not the same as a Whitney stratification, the latter being a strictly stronger notion; see \cite[\S5]{NoTS} for further details.  Indeed, the strong notion of a Whitney stratification will not be required for the purposes of this paper.
\end{Rk}

A key source of stratifications is provided by the following definition:

\begin{Defn}
Let \(E\) be an orbifold.  For each isomorphism class of finite groups \([\Ga]\), the set:
\ew
E(\Ga) = \lt\{e \in E ~\m|~ \Ga_e \in [\Ga]\rt\}
\eew
is either empty or its connected components are submanifolds of \(E\), with \(E([\bf{1}])\) being always an open and dense subset of \(E\).  Take \(E_i\) to be an enumeration of the connected components of the \(E([\Ga])\) as \([\Ga]\) varies.  Then one can show that (at least when \(E\) is compact) there are only a finite number of strata \(E_i\) and that \(\Si_{can} = \lt\{E_i\rt\}\) defines a stratification of \(E\) known as the canonical stratification of \(E\).  One may verify that this stratification satisfies the condition of the frontier.
\end{Defn}

\begin{Ex}
Consider:
\ew
\Ga = \lt\{ \bpm 1 & 0\\ 0 & 1 \epm, \bpm -1 & 0\\ 0 & 1 \epm, \bpm 1 & 0\\ 0 & -1 \epm, \bpm -1 & 0\\ 0 & -1 \epm \rt\} \pc \GL(2;\bb{R}).
\eew
The quotient \(E = \lqt{\bb{R}^2}{\Ga}\) is a 2-dimensional orbifold.  The canonical stratification of \(E\) is given by:
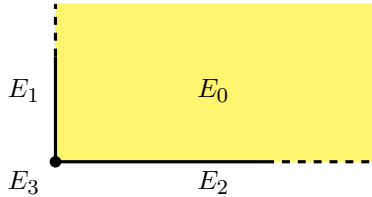
\begin{center}
\begin{tikzpicture}[scale = 1.4]
\filldraw[yellow!70] (0,0) rectangle (3,1.5);
\draw[black!100, very thick] (0,0) -- (2,0);
\draw[black!100, very thick] (0,0) -- (0,1);
\draw[black!100, dashed, very thick] (2,0) -- (3,0);
\draw[black!100, dashed, very thick] (0,1) -- (0,1.5);
\node [below] at (1.5,0.875) {\(E_0\)};
\node [below] at (-0.3,0.875) {\(E_1\)};
\node [below] at (1.5,0) {\(E_2\)};
\node [below] at (-0.3,0) {\(E_3\)};
\filldraw[black!100] (0,0) circle (0.05);
\end{tikzpicture}
\end{center}
Note that whilst each \(E_i\) is a submanifold of \(E\), the singular locus \(S = E_1 \cup E_2 \cup E_3\) is not even a suborbifold of \(E\): indeed, in the natural (global) orbifold chart for \(E\), \(S\) corresponds to the subset \(\bb{R} \x \{0\} \cup \{0\} \x \bb{R} \pc \bb{R}^2\), which is not a linear subspace.
\end{Ex}

\subsection{Vector bundles over orbifolds}

\begin{Defn}[{Cf.\ \cite[\S14.1]{THKLFPFftScDO}}]
Let \(\pi:E \to B\) be a smooth map of orbifolds.  An (orbifold) vector bundle chart \(\Th_b\) for \(\pi\) about \(b \in B\) is the data of:
\begin{itemize}
\item A chart \(\Xi_b\) for \(B\) centred at \(b\);
\item A chart \(\Xi_e\) for \(E\) centred at a suitable \(e \in \pi^{-1}(b)\);
\item A local lift \(\tld{\pi}:\tld{U}_e \to \tld{U}_b\) for \(\pi\) and a homomorphism \(\ka_\pi:\Ga_e \to \Ga_b\) as in \dref{c-inf-map},
\end{itemize}
such that:
\begin{enumerate}
\item \(\tld{\pi}:\tld{U}_e \to \tld{U}_b\) is a vector bundle of some rank \(k\) and \(0 \in \tld{U}_e\) is the zero of the fibre over \(0 \in \tld{U}_b\);
\item \(\ka_\pi: \Ga_e \to \Ga_b\) is an isomorphism and \(\Ga_e \cong \Ga_b\) acts on \(\tld{U}_e\) via vector bundle automorphisms.
\end{enumerate}
An embedding of a chart \(\Th_{b_1} = (\Xi_{b_1},\Xi_{e_1},\tld{\pi}_1,(\ka_\pi)_1)\) into \(\Th_{b_2} = (\Xi'_{b_2},\Xi'_{e_2}, \tld{\pi}_2, (\ka_\pi)_2)\) is the data of embeddings of orbifold charts \((\io_{b_1b_2},\la_{b_1b_2}):\Xi_{b_1} \emb \Xi_{b_2}\) and \((\io_{e_1e_2},\la_{e_1e_2}):\Xi_{e_1} \emb \Xi'_{e_2}\) such that the induced equivariant commutative square:
\ew
\bcd[column sep = 13mm, row sep = 3mm]
\tld{U}_{e_1} \ar[rr, ^^22 \mla{\io_{e_1e_2}} ^^22] \ar[dd, ^^22 \mla{\tld{\pi}_2} ^^22] & & \tld{U}_{e_2} \ar[dd, ^^22 \mla{\tld{\pi}_2} ^^22]\\
& &\\
\tld{U}_{b_1} \ar[rr, ^^22 \mla{\io_{b_1b_2}} ^^22] & & \tld{U}_{b_2}
\ecd
\eew
is a bundle isomorphism (and thus the ranks of the two bundles are equal).   In particular, given a point \(b \in B\), the distinguished point \(e \in \pi^{-1}(b)\) is unique and does not depend on the choice of chart. 

As for orbifolds, two vector bundle charts \(\Th_{b_1}\) and \(\Th_{b_2}\) are called compatible if for all \(b \in U_{b_1} \cap U_{b_2}\), there exists a chart \(\Th_b\) centred at \(b\) which embeds into both \(\Th_{b_1}\) and \(\Th_{b_2}\).  An (orbifold) vector bundle is then a smooth map \(\pi:E \to B\) of orbifolds together with a maximal atlas of compatible vector bundle charts.  Note that \(\pi\) has a well defined rank \(k\).  The fibre \(E_b\) over a point \(b \in B\) is naturally identified with the space \(\lqt{\bb{R}^k}{\Ga_b}\).  A section of \(E\) is then simply a continuous map \(X: B \to E\) such that for each chart \(\Th_b = (\Xi_b,\Xi_e,\tld{\pi},\ka_\pi)\) for \(E\), there is a smooth, \(\Ga_b \cong \Ga_e\)-equivariant, section \(\tld{X}: \tld{U}_b \to \tld{U}_e\) such that the following diagram commutes:
\ew
\bcd[column sep = 13mm, row sep = 3mm]
\tld{U}_b \ar[rr, ^^22 \mla{\tld{X}} ^^22] \ar[dd, ^^22 \mla{\tld{\ch}_b} ^^22] & & \tld{U}_e \ar[dd, "\mla{\tld{\ch}_e}"]\\
& &\\
U_b \ar[rr, ^^22 \mla{X} ^^22] &  & U_e
\ecd
\eew
Note in particular, that \(\tld{X}|_0\) is invariant under the action of \(\Ga_e \cong \Ga_b\) on the fibre \(\tld{U}_e|_0\).  Thus \(X|_b\) can be regarded as an element of the subspace of \(\bb{R}^k\) upon which \(\Ga_b\) acts trivially, denoted \(Fix_{\Ga_b}(\bb{R}^k)\), which, in turn, can be regarded as a subspace of \(\lqt{\bb{R}^k}{\Ga_b}\).  Finally, let \(\pi: E \to B\) be a vector bundle and let \(F \cc E\) be a suborbifold.  The smooth map \(\pi|_F:F \to B\) is called a sub-vector bundle of \(\pi:E \to B\) if for any chart \(\Th_b = (\Xi_b,\Xi_e,\tld{\pi},\ka_\pi)\) for \(\pi\), the subset \(\tld{\ch}_e^{-1}(U_e \cap F) \pc \tld{U}_e\) is a \(\Ga_b \cong \Ga_e\)-invariant sub-vector bundle of \(\tld{\pi}:\tld{U}_e \to \tld{U}_b\).
\end{Defn}

Let \(E\) be an \(n\)-orbifold.  Given any chart \(\Xi = \lt(U, \Ga, \tld{U}, \ch\rt)\) for \(E\), the action of \(\Ga\) on \(\tld{U}\) naturally lifts to an action of \(\Ga\) on \(\T\tld{U}\) by bundle automorphisms.  Given a second chart \(\Xi' = (U',\Ga',\tld{U}',\ch')\) embedding into \(\Xi\), the map \(\tld{U}' \emb \tld{U}\) induces an equivariant embedding \(\T\tld{U}' \emb \T\tld{U}\).  Define:
\ew
\T E = \rqt{\lt[\coprod_{\Xi} \lt(\lqt{\T\tld{U}}{\Ga}\rt)\rt]}{\tl}
\eew
where the quotient by \(\tl\) denotes that one should glue along the embeddings \(\T\tld{U}' \emb \T\tld{U}\).  The resulting space \(\T E\) is an orbifold vector bundle over \(E\) and is termed the tangent bundle of \(E\).  Given \(e \in E\), the tangent space at \(e\), denoted \(\T_eE\), is the preimage of \(e\) under the map \(\T E \to E\) and may be identified with the quotient space \(\lqt{\bb{R}^n}{\Ga_e}\), where \(\Ga_e\) is the orbifold group at \(e\).  In a similar way, one can define the cotangent bundle of an orbifold, tensor bundles, bundles of exterior forms, etc., denoted in the usual way.

Now let \(\pi:F \to E\) be such a vector bundle over \(E\) of rank \(k\) and let \(\mc{A} \cc \bb{R}^k\) be a \(\GL(k;\bb{R})\)-invariant subset.  For each \(e \in E\), the subset \(\lqt{\mc{A}}{\Ga_e} \cc \lqt{\bb{R}^k}{\Ga_e}\) is well-defined and gives rise to a subset of \(\pi^{-1}(e)\) under the identification \(\pi^{-1}(e) \cong \lqt{\bb{R}^k}{\Ga_e}\).  As \(e \in E\) varies, this defines a sub-bundle of \(F\).

\begin{Defn}
Let \(E\) be an \(n\)-orbifold and let \(\ss{2}\T^*E\) denote the vector bundle of symmetric bilinear forms on \(\T E\).  The subspace \(\ss[+]{2}\lt(\bb{R}^n\rt)^* \pc \ss{2}\lt(\bb{R}^n\rt)^*\) of inner-products is \(\GL(n;\bb{R})\)-invariant and hence one can define the sub-bundle \(\ss[+]{2}\T^*E\) of positive-definite inner-products on \(E\).  A(n) (orbifold) Riemannian metric on \(E\) is simply a section \(g\) of \(\ss[+]{2}\T^*E\).  Likewise, the set \(\ss[\ge0]{2}\lt(\bb{R}^n\rt)^*\) of non-negative-definite symmetric bilinear forms is also \(\GL(n;\bb{R})\)-invariant.  This gives rise to the bundle \(\ss[\ge0]{2}\T^*E\) and sections \(h\) of this bundle are called (orbifold) Riemannian semi-metrics.  Given Riemannian (semi-)metrics \(g\) and \(h\), write \(g \ge h\) if \(g - h\) is non-negative definite, i.e.\ if \(g - h\) defines a \rsm.
\end{Defn}

Given a Riemannian (semi)-metric \(g\) on \(E\), recall that for each \(e \in E\), \(g|_e\) can be regarded as an element of the space \(Fix_{\Ga_e}\lt(\ss{2}\lt(\bb{R}^n\rt)^*\rt) \cc \lqt{\ss{2}\lt(\bb{R}^n\rt)^*}{\Ga_e}\).  Thus, given \(u, u' \in \T_e E \cong \lqt{\bb{R}^n}{\Ga_e}\), the quantity \(g(u,u')\) is well-defined; write \(g(u)\) as a shorthand for \(g(u,u)\).  Given a \(C^1\) curve \(\ga:[a,b] \to X\), one can define:
\ew
\ell^g(\ga) = \bigintsss_{[a,b]} g\lt(\dot{\ga}\rt) \dd \cal{L}
\eew
where \(\cal{L}\) denotes the Lebesgue measure on \([a,b]\).  If \(\ga\) is merely piecewise-\(C^1\), i.e.\ \(\ga = \ga_1 \1 \ga_2 \1 ... \1 \ga_k\) is the concatenation of \(C^1\)-curves, define:
\ew
\ell^g(\ga) = \sum_{i = 1}^k \ell^g(\ga_i).
\eew
Then \(\ell^g\) defines a (weak) length structure on the set of piecewise-\(C^1\) paths in \(E\), and thus a (semi)-metric on \(E\), denoted \(d^g\).

\subsubsection{Stratified distributions}

A distribution on an orbifold \(E\) is simply a sub-vector bundle \(\mc{D}\) of \(\T E\).  Given a Riemannian metric \(g\) on \(E\), define the orthocomplement \(\mc{D}^\bot\) to \(\mc{D}\) via the formula:
\ew
\mc{D}^\bot|_e = \lt\{ u \in \T_eE ~\m|~ g(u,u') = 0 \text{ for all } u' \in \mc{D}|_e\rt\}.
\eew
Then \(\mc{D}^\bot\) is also a distribution over \(E\).  Indeed recall that, locally, \(\mc{D}\) is given by \(\lqt{\tld{\mc{D}}}{\Ga_e}\) for some \(\Ga_e\)-invariant distribution \(\tld{\mc{D}} \pc \T \tld{U}\), where \(\tld{U}\) is a local chart for \(E\), and \(g\) is induced by a \(\Ga_e\)-invariant Riemannian metric \(\tld{g}\) over \(\tld{U}\); from this the result is clear.

Now let \(\Si\) be a stratification on \(E\).  Even in the case where \(E\) is a manifold, a general distribution \(\mc{D}\) can be `incompatible' with \(\Si\) in the following sense:
\begin{Ex}
Consider the distribution \(\mc{D}\) over \(E = \bb{R}^2\) given by \(\mc{D} = \<\del_1 + x^1\del_2\?\) and the stratification \(\Si\) of \(E\) given by:
\ew
E_0 = \lt.\bb{R}^2 \m\osr (\bb{R} \x \{0\}) \rt. \et E_1 = \bb{R} \x \{0\} \cong \bb{R}.
\eew
Then \(\mc{D} \cap \T E_1\) is not a distribution over \(E_1\), since over the non-zero points of \(E_1 \cong \bb{R}\), \(\mc{D}\) only intersects \(\T E_1\) along its zero-section, however over the point 0 the fibres of \(\mc{D}\) and \(\T E_1\) coincide.
\end{Ex}
This potential for incompatibility motivates the following definition, which cannot (to the author's knowledge) be found in the literature:

\begin{Defn}
Let \((E,\Si = \lt\{E_i\rt\})\) be a stratified orbifold.  A distribution \(\mc{D}\) on \(E\) is termed stratified if \(\mc{D}_i = \mc{D} \cap \T E_i \cc \T E_i\) is a distribution over \(E_i\), for all \(i\).
\end{Defn}

I remark that, if \(\mc{D}\) is stratified, then for every Riemannian metric \(g\), the orthocomplement \(\mc{C} = \mc{D}^\bot\) is also stratified.  Indeed, for each \(i\):
\ew
\mc{C} \cap \T E_i = \lt(\mc{D}_i\rt)^\bot,
\eew
where the orthocomplement is defined using Riemannian metric \(g|_{E_i}\) on the stratum \(E_i\).

\section{Stratified fibrations between orbifolds}\label{submersions}

Following \cite[\S3.2]{ItO}, I call a smooth map \(f: E \to B\) a submersion if for all \(e \in E\), there exist a chart \(\Xi_e\) centred at \(e\), a chart \(\Xi_{f(e)}\) centred at \(f(e)\), and a local representation \((\tld{f},\ka_f)\) in these charts such that \(\tld{f}\) is submersive and \(\ka_f\) is surjective.  The following result may not, to the author's knowledge, be found in the literature.  It is the analogue of Ehresmann's Theorem for orbifolds (see e.g.\ \cite[Thm.\ 9.3]{HT&CAG} for the classical statement):

\begin{Prop}\label{loc-triv}
Let \(E\), \(B\) be orbifolds and let \(\pi:E \to B\) a proper, surjective submersion.  Let \(\ker\dd\pi\) be the vertical distribution of \(\pi\), pick a Riemannian metric \(g\) on \(E\) and let \(\mc{C} = \ker\dd\pi^\bot\) be the corresponding horizontal distribution.

\begin{enumerate}
\item Let \(\ga: (-1,1) \to B\) be an embedded curve.  Then \(\pi^{-1}(\ga(-1,1)) \cc E\) is a suborbifold, denoted \(E_\ga\).  Write \(E_0 = \pi^{-1}(\ga(0))\).  Then there is an orbifold diffeomorphism:
\ew
E_\ga \cong E_0 \x \ga(-1,1)
\eew
identifying \(\pi\) with projection onto the second factor and \(\mc{C}\) with the product connection;
\item Let \(U \cc B\) be an open ball and write \(E_U = \pi^{-1}(U) \cc E\).  Write \(b\) for the centre of the ball \(U\) and write \(E_b = \pi^{-1}(b)\).  Then \(E_U\) is a suborbifold of \(E\) and there is an orbifold diffeomorphism:
\ew
E_U \cong E_b \x U
\eew
identifying \(\pi\) with projection onto the second factor (although the identification does not identify \(\mc{C}\) with the product distribution, in general).
\end{enumerate}
\end{Prop}

To prove \pref{loc-triv}, I begin by recording the following equivariant version of the Implicit Function Theorem:
\begin{Prop}\label{EIFT}
Let \(\Ga_i \pc \GL(n_i;\bb{R})\) be finite subgroups, \(i = 1,2\), and let \(\tld{U} \cc \bb{R}^{n_1}\) be a \(\Ga_1\)-invariant open neighbourhood of \(0\).   Suppose one is given a group homomorphism \(\io:\Ga_1 \to \Ga_2\) and a smooth map \(f:\tld{U} \to \bb{R}^{n_2}\) which is \(\io\)-equivariant, which maps \(0 \in \bb{R}^{n_1} \mt 0 \in \bb{R}^{n_2}\) and has surjective derivative at \(0\).  Write \(\bb{K} = \ker (\dd f|_0)\), a \(\Ga_1\)-invariant subspace of \(\bb{R}^{n_1}\) of dimension \(n_1 - n_2\) and let \(\bb{T}\) be any \(\Ga_1\)-invariant complementary subspace to \(\bb{K}\) in \(\bb{R}^{n_1}\).

Then (shrinking \(\tld{U}\) if necessary) there is a \(\Ga_1\)-equivariant diffeomorphism:
\ew
F: \tld{U} \cc \bb{K} \x \bb{T} \to F\lt(\tld{U}\rt) \cc \bb{K} \x \bb{R}^{n_2}
\eew
(where \(\Ga_1\) acts on \(\bb{R}^{n_2}\) via the map \(\io:\Ga_1 \to \Ga_2\)) such that \(\dd F|_0\) identifies \(\bb{T}\) with \(\bb{R}^{n_2}\) and such that the diagram:
\ew
\bcd
\tld{U} \cc \bb{R}^{n_1} \ar[rr, " \mla{F} "] \ar[rd, " \mla{f} "'] & & F\lt(\tld{U}\rt) \cc \bb{K} \x \bb{R}^{n_2} \ar[ld, " \mla{\pi_2} "]\\
& \bb{R}^{n_2} &
\ecd
\eew
commutes and is equivariant.  In particular, if \(\io\) is surjective, then the bottom arrow in the following diagram is an isomorphism:
\ew
\bcd[column sep = 10mm, row sep = 3mm]
0 \x \bb{T} \ar[rr, " \mla{\dd f|_0} "] \ar[dd, " \mla{proj} "] & & \bb{R}^{n_2} \ar[dd, " \mla{proj} "]\\
& &\\
\lqt{(0 \x \bb{T})}{\Ga_1} \ar[rr] & & \lqt{\bb{R}^{n_2}}{\Ga_2}
\ecd
\eew
\end{Prop}
The proof is a simple application of the Inverse Function Theorem, and so is omitted here.

\begin{proof}[Proof of \pref{loc-triv}]
That \(E_\ga\) and \(E_U\) are suborbifolds of \(E\) follows at once from the local description of \(\pi\) afforded by \pref{EIFT}.  Moreover, \pref{EIFT} shows that for each \(e \in E\), \(\dd\pi|_e\) defines an isomorphism \(\mc{C}_e \to \T_{\pi(e)}B\).

Now consider (1).  Choose a point \(e \in E_0\).  The derivative of \(\ga\) defines a natural map \(\dot{\ga}: (-1,1) \to \T B|_\ga\).  Using \(\dd\pi\), on may lift this uniquely to a map \((-1,1) \to \mc{C}\); integrating this vector field along \((-1,1)\) defines the horizontal lift of \(\ga\) starting from \(e\), denoted \(\ga_e\).  Now define a map:
\ew
\bcd[row sep = 0pt]
E_0 \x (-1,1) \ar[r]& E_\ga\\
(e,t) \ar[r, maps to]& \ga_e(t)
\ecd
\eew
One may verify that this is the required diffeomorphism.  Given (1), (2) follows as for manifolds, by trivialising \(\pi\) along radial paths emanating from \(b\).

\end{proof}

In general, the requirement that the homomorphisms \(\ka_f\) be surjective can be rather strict.  The following definition relaxes this condition:
\begin{Defn}
A smooth map of orbifolds \(f: E \to B\) is a weak submersion if for all \(e \in E\), there exists a chart \(\Xi_e\) centred at \(e\), a chart \(\Xi_{f(e)}\) centred at \(f(e)\), and a local representation \((\tld{f},\ka_f)\) in these charts such that \(\tld{f}\) is submersive. (In particular, \(\ka_f\) is not assumed to be surjective.)
\end{Defn}

Clearly \pref{loc-triv} does not apply to weak submersions in general, as the following example illustrates:

\begin{Ex}\label{FibrEx}
Let \(E = \lqt{\bb{T}^2}{\{\pm1\}}\), \(B = \lqt{[\bb{T}^1 \x \{0\}]}{\{\pm1\}}\) and let \(\pi: E \to B\) denote the canonical projection.  Then \(E\) is the `pillowcase', homeomorphic to a 2-sphere, with singular points precisely the four corners of the `pillowcase' and \(B\) is a closed interval, with singular points precisely the endpoints of the interval.
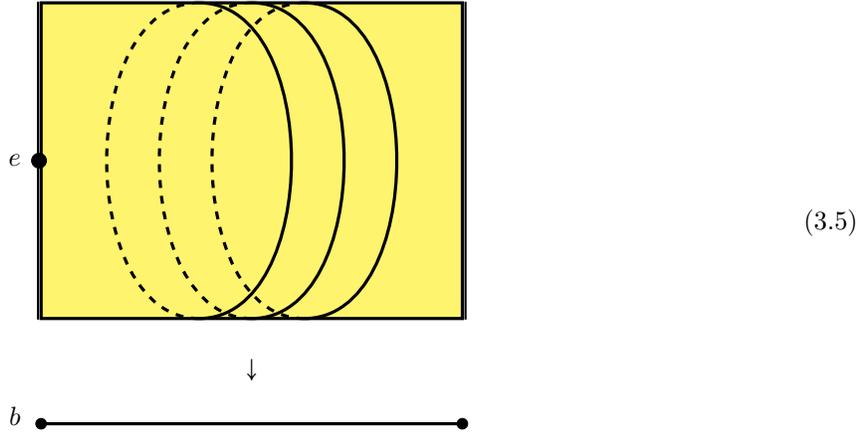
\e\label{Fibr-Ex}
\begin{tikzpicture}[scale = 1.4]
\filldraw[yellow!70] (0,0) rectangle (4,3);
\draw[black!100, very thick] (0,0) rectangle (4,3);
\draw[black!100, dashed, very thick] (1.5,0) to [out=180, in=180] (1.5,3);
\draw[black!100, dashed, very thick] (2,0) to [out=180, in=180] (2,3);
\draw[black!100, dashed, very thick] (2.5,0) to [out=180, in=180] (2.5,3);
\draw[black!100, very thick] (1.5,0) to [out=0, in=0] (1.5,3);
\draw[black!100, very thick] (2,0) to [out=0, in=0] (2,3);
\draw[black!100, very thick] (2.5,0) to [out=0, in=0] (2.5,3);
\draw[black!100, thick] (-0.03,-0.01)--(-0.03,3.01);
\draw[black!100, thick] (4.03,-0.01)--(4.03,3.01);

\node [below] at (2,-0.3) {\(\downarrow\)};

\draw[black!100, very thick] (0,-1) -- (4,-1);
\filldraw[black!100] (0,-1) circle (0.05);
\filldraw[black!100] (4,-1) circle (0.05);

\node [below] at (-0.25,-0.75) {\(b\)};

\filldraw[black!100] (-0.02,1.5) circle (0.07);
\node [below] at (-0.25, 1.65) {\(e\)};
\end{tikzpicture}
\ee
\(\pi\) is a surjective, proper, weak submersion, however whilst the preimage of a smooth point in \(B\) is topologically a circle, the preimage of a singular point is topologically a closed interval.  Thus \(\pi\) is not a locally-trivial fibration (or even a Serre fibration, since the homotopy groups of its fibres are not constant over the connected base space).  Note also that the points \(e\) and \(b\) marked in the diagram satisfy \(\Ga_e = \bf{1}\) and \(\Ga_b = \rqt{\bb{Z}}{2}\).  Thus \(\dd\pi_e: \mc{C}_e \to \T_bB\) is not an isomorphism (in fact, it is a 2:1 quotient).
\end{Ex}

However, a stratified version of \pref{loc-triv} still holds for weak submersions:
\begin{Cor}\label{ind-strat}
Let \(E\) and \(B\) be orbifolds, let \(\pi: E \to B\) be a proper, surjective, weak submersion and let \(\Si(B) = \lt\{B_i\rt\}\) be a stratification of \(B\).  For each \(B_i\), the subset \(\pi^{-1}\lt(B_i\rt) \cc E\) is a suborbifold and the restriction of \(\pi\) to \(\pi^{-1}\lt(B_i\rt) \cc E\) defines a submersion onto the submanifold \(B_i \pc B\) in the usual orbifold sense.  Let \(\lt\{E_i^j\rt\}_j\) denote the strata in the canonical stratification of \(\pi^{-1}\lt(B_i\rt)\).  Then \(\pi: E_i^j \to B_i\) is a surjective submersion for all \(i,j\).  Moreover, the collection \(\lt\{E_i^j\rt\}_{i,j}\) defines a stratification of \(E\), denoted \(\Si(\pi,E,B)\), \wrt\ which the distribution \(\mc{D} = \ker(\dd\pi)\) is a stratified distribution.  In particular, the orthocomplement of the vertical distribution (\wrt\ any Riemannian metric) is also stratified.
\end{Cor}

\begin{proof}
Since any weak submersion from an orbifold to a manifold is (trivially) a submersion (noting that any group homomorphism to the trivial group is surjective), the corollary follows simply by working in appropriate local charts as in the proof of \pref{loc-triv}.  For example, to see that \(\mc{D}\) is stratified \wrt\ \(\Si(\pi,E,B)\), pick a stratum \(B_i\) in \(B\), a stratum \(E_i^j\) in \(\pi^{-1}\lt(B_i\rt)\), a point \(e \in E_i^j\) and write \(b = \pi(e)\).  Let \(U\) be an open neighbourhood of \(b\) in \(B_i\) and recall from \pref{loc-triv} that \(E_U = \pi^{-1}(U)\) can be identified with \(\pi^{-1}(b) \x U\).  Letting \(S\) denote the canonical stratum of \(\pi^{-1}(b)\) containing \(e\) and writing \(W\) for a small open neighbourhood of \(e\) in \(S\), then (shrinking \(U\) if necessary) it follows that an open neighbourhood of \(e\) in \(E_i^j\) is given by \(W \x U\), hence \(\mc{D} \cap \T E_i^j\) is locally given by \(\T W \ds 0 \pc \T(W \x U)\) and whence \(\mc{D}\) is stratified, as claimed.

\end{proof}

\begin{Rk}\label{strat-rk}
Note that \(E_i^j\) need not be given globally as the product of a canonical stratum of \(\pi^{-1}(b)\) with \(B_i\): indeed, take \(E\) to be the M\"{o}bius band (viewed as an orbifold with singular set precisely its boundary), let \(B = S^1\) and consider the usual projection \(\pi:E \to B\).  Then the stratification of \(\M\) induced by \(\pi\) is simply its canonical stratification (which has two strata), whereas the canonical stratification of the preimage of any point in \(B\) has 3 strata, the two endpoints being different strata of the preimage, but belonging to the same stratum of \(E\).  More generally, writing \(\{S_j\}_j\) for the canonical stratification of \(\pi^{-1}(b)\), then the stratification of \(E_U\) given by \(\{S_j \x U\}_j\) is a refinement of the stratification \(\{E_i^j \cap E_U\}_j\).  Phrased differently, the strata \(E_i^j\) of \(\pi^{-1}(B_i)\) are stable under the `horizontal transport maps' used in the proof of \pref{loc-triv}.
\end{Rk}

\begin{Ex}
Return to \eref{Fibr-Ex}.  The stratification \(\Si(\pi,E,B)\) is depicted below:
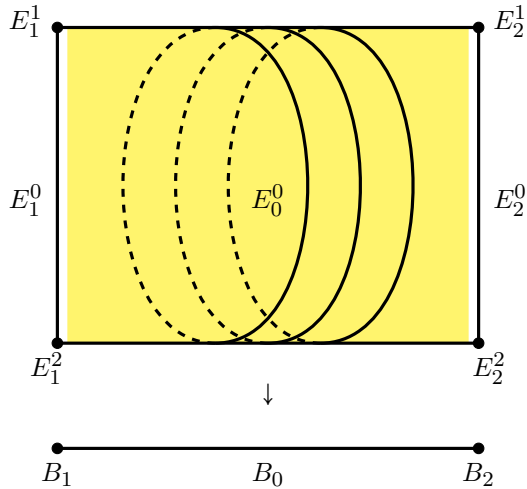
\begin{center}
\begin{tikzpicture}[scale = 1.4]
\filldraw[yellow!70] (0.1,0) rectangle (3.9,3);
\draw[black!100, very thick] (0,0) rectangle (4,3);
\draw[black!100, dashed, very thick] (1.5,0) to [out=180, in=180] (1.5,3);
\draw[black!100, dashed, very thick] (2,0) to [out=180, in=180] (2,3);
\draw[black!100, dashed, very thick] (2.5,0) to [out=180, in=180] (2.5,3);
\draw[black!100, very thick] (1.5,0) to [out=0, in=0] (1.5,3);
\draw[black!100, very thick] (2,0) to [out=0, in=0] (2,3);
\draw[black!100, very thick] (2.5,0) to [out=0, in=0] (2.5,3);

\node [below] at (2,-0.3) {\(\downarrow\)};

\draw[black!100, very thick] (0,-1) -- (4,-1);
\filldraw[black!100] (0,-1) circle (0.05);
\filldraw[black!100] (4,-1) circle (0.05);

\node [below] at (2,-1.05) {\(B_0\)};
\node [below] at (0,-1.05) {\(B_1\)};
\node [below] at (4,-1.05) {\(B_2\)};

\node [below] at (2,1.6) {\(E_0^0\)};
\node [below] at (-0.3,1.6) {\(E_1^0\)};
\node [below] at (-0.3,3.3) {\(E_1^1\)};
\node [below] at (-0.1,0) {\(E_1^2\)};
\node [below] at (4.3,1.6) {\(E_2^0\)};
\node [below] at (4.3,3.3) {\(E_2^1\)};
\node [below] at (4.1,0) {\(E_2^2\)};

\filldraw[black!100] (0,0) circle (0.05);
\filldraw[black!100] (0,3) circle (0.05);
\filldraw[black!100] (4,0) circle (0.05);
\filldraw[black!100] (4,3) circle (0.05);
\end{tikzpicture}
\end{center}
\end{Ex}

\section{Stratified Riemannian, semi-Riemannian and quasi-Finslerian structures on orbifolds}\label{Strat-data}

The aim of this section is to introduce two new generalisations of Riemannian (semi)-metrics on orbifolds.  Firstly, the usual definition will be generalised to allow Riemannian (semi)-metrics on a stratified orbifold \((E,\Si)\) which are continuous on each stratum and `boundedly discontinuous' across strata.  Secondly, the definition of Riemannian metrics will be generalised by replacing the fibrewise inner product on the tangent bundle of E (the usual notion of a Riemannian metric) with a fibrewise quasi-norm on \(\T E\).  Both generalisations will be required to state and prove the main theorem \tref{Cvgce-Thm-1-FV}.

I begin by recalling the following definition \cite[\S15.10]{TVSI}:
\begin{Defn}\label{qnorm}
Let \(\bb{A}\) be a real vector space. A quasinorm on \(\bb{A}\) is a map \(\mc{L}:\bb{A} \to \bb{R}\) satisfying the following three properties:
\begin{enumerate}
\item For all \(a \in \bb{A}\): \(\mc{L}(a) \ge 0\), with equality \iff\ \(a=0\) (\(\mc{L}\) is `positive-definite');

\item For all \(\la\in\bb{R}\), \(a\in\bb{A}\):
\ew
\mc{L}(\la\cdot a) = |\la|\cdot\mc{L}(a);
\eew

\item There exists some \(k = k(\mc{L})>0\) such that for all \(a,a' \in \bb{A}\):
\ew
\mc{L}(a+a') \le k(\mc{L}(a) + \mc{L}(a')).
\eew
\end{enumerate}
Note that in the case \(k = 1\), this reduces to the definition of a norm.
\end{Defn}

In this paper, I restrict attention to continuous quasinorms.  In this case, condition (3) above becomes automatic:
\begin{Prop}
Let \(\bb{A}\) be a finite-dimensional real vector space and let \(\mc{L}: \bb{A} \to \bb{R}\) be a continuous map satisfying conditions (1) and (2) from \dref{qnorm}.  Then \(\mc{L}\) is a quasinorm.
\end{Prop}

\begin{proof}
Consider the continuous map:
\ew
f:(\bb{A} \x \bb{A})\backslash\{0\} &\to [0,\0)\\
(a,a') &\mt \frac{\mc{L}(a+a')}{\mc{L}(a) + \mc{L}(a')}
\eew
(Note that \(f\) is well-defined by condition (1) in \dref{qnorm}.)  For a contradiction, suppose \(f\) is unbounded and pick a sequence \((a_i,a_i') \in (\bb{A} \x \bb{A})\backslash\{0\}\) such that \(f(a_i,a_i') \to \0\) as \(i\to\0\). Choose some norm \(\|-\|\) on \(\bb{A}\) and consider the new sequence:
\ew
(\fr{a}_i,\fr{a}_i') = \lt(\frac{a_i}{\|a_i\| + \|a_i'\|}, \frac{a_i'}{\|a_i\| + \|a_i'\|}\rt) \in \bb{A} \x \bb{A}.
\eew
Clearly \((\fr{a}_i,\fr{a}_i')\) is bounded in the norm \(\|-\|\) and hence converges subsequentially to some \((\fr{a},\fr{a}') \in \bb{A} \x \bb{A}\) (since \(\bb{A}\) is finite-dimensional). Moreover, by construction, the sequence \((\fr{a}_i,\fr{a}_i')\) satisfies \(\|\fr{a}_i\| + \|\fr{a}_i'\| = 1\) and thus \(\|\fr{a}\| + \|\fr{a}'\| = 1\).  Hence \((\fr{a},\fr{a}') \in (\bb{A} \x \bb{A})\osr\{0\}\), and thus \(f(\fr{a},\fr{a}')\) is well-defined and finite.  By condition (2) in \dref{qnorm}, \(f\) satisfies \(f(\la-,\la-) = f(-,-)\) for any \(\la\ne0\).  Therefore:
\ew
f(a_i,a_i') = f\lt(\frac{a_i}{\|a_i\| + \|a_i'\|}, \frac{a_i'}{\|a_i\| + \|a_i'\|}\rt) = f(\fr{a}_i,\fr{a}_i') \oto{\text{subsequentially}} f(\fr{a},\fr{a}') < \0,
\eew
contradicting the fact that \(f(a_i,a_i') \to \0\) as \(i\to \0\).  Thus \(f\) is bounded and \(\mc{L}\) is a quasinorm.

\end{proof}

\begin{Defn}
Let \(E\) be an manifold.   A quasi-Finslerian structure on \(E\) shall here denote a continuous map:
\ew
\mc{L}:\T E \to \bb{R}
\eew
such that the restriction of \(\mc{L}\) to any fixed tangent space is a (continuous) quasinorm.  (Note that in the case where \(\mc{L}\) is smooth and a fibrewise norm, this recovers the usual definition of a Finslerian structure.)
\end{Defn}

Given a symmetric bilinear form \(g\) on a real vector space \(\bb{A}\), I shall refer to the kernel of the linear map \(()^\flat: a \in \bb{A} \mt g(a,-) \in \bb{A}^*\) simply as the kernel of \(g\).  Using this notation, I now introduce the required generalisations of Riemannian metrics to stratified orbifolds:

\begin{Defn}\label{sRsm}
Let \((E,\Si = \{E_i\}_i)\) be a stratified orbifold.
\begin{itemize}
\item A stratified Riemannian metric \(\h{g} = \{g_i\}_i\) on \(E\) is the data of a Riemannian metric \(g_i\) on each stratum \(E_i\) satisfying the extendibility condition that for each \(i\), there exists a continuous orbifold Riemannian metric \(\overline{g}_i\) on \(E\) whose tangential component along \(E_i\) is \(g_i\).

\item A stratified Riemannian semi-metric \(\h{g} = \{g_i\}_i\) on \(E\) is the data of a Riemannian semi-metric \(g_i\) on each stratum \(E_i\) satisfying the analogous condition that for each \(i\), there exists a continuous orbifold Riemannian semi-metric \(\overline{g}_i\) on \(\T E\) whose tangential component along \(E_i\) is \(g_i\).  If, in addition, \(\mc{D}\) is a stratified distribution on \(E\), then \(\h{g}\) is regular \wrt\ \(\mc{D}\) if for each \(i = 0,...,n\), the kernel of the Riemannian semi-metric \(\overline{g}_i\) is precisely the distribution \(\mc{D}\).  In particular, this implies that the kernel of the Riemannian semi-metric \(g_i\) on \(E_i\) is precisely \(\mc{D}_i = \mc{D} \cap \T E_i\).

\item A stratified quasi-Finslerian structure on \(E\) is the data of a quasi-Finslerian structure \(\mc{L}_i\) on each \(E_i\) satisfying the property that for every continuous orbifold Riemannian metric \(h\) on \(\T E\) and each index \(i\), there exists a continuous function \(C:\overline{E_i} \to (0,\0)\) such that
\e\label{LE-qF-defn}
\frac{1}{C} \|-\|_h \le \mc{L}_i \le C\|-\|_h \on E_i.
\ee
\end{itemize}
\end{Defn}

\begin{Rks}~

\begin{itemize}
\item
Any two continuous quasinorms \(\mc{L}\) and \(\mc{L}'\) on a finite-dimensional vector space \(\bb{A}\) are Lipschitz equivalent.  Indeed, let \(\mc{S} \pc \bb{A}\) be the unit sphere \wrt\ some norm on \(\bb{A}\); then \(\frac{\mc{L}}{\mc{L}'}: \mc{S} \to (0,\0)\) is well-defined and continuous, and hence has compact image (as \(\mc{S}\) is compact).  In light of this, the function \(C\) in \eref{LE-qF-defn} automatically exists on \(E_i\).  The significance of \eref{LE-qF-defn} is that \(C\) can be extended continuously over the boundary of \(E_i\), i.e.\ over the set \(\lt.\overline{E_i}\m\osr E_i\rt.\).

\item The extendibility condition for stratified Riemannian metrics can alternatively be stated as follows: for every subset \(K \cc E_i\) which is relatively compact in \(E\):
\begin{enumerate}
\item \(g_i|_K\) is uniformly continuous;
\item \(g_i\) is uniformly Lipschitz equivalent to \(g|_{E_i}\) for any continuous Riemannian metric \(g\) on \(E\).
\end{enumerate}
The reader will note, by contrast, that condition (1) is not imposed on stratified quasi-Finlserian structures.  This extra condition is required for stratified Riemannian metrics to facilitate some technical steps in \sref{Gen-Coll}.
\end{itemize}
\end{Rks}

The stratified structures defined in \dref{sRsm} naturally induce (semi)-metrics on the underlying orbifold \(E\), in the following way:
\begin{Defn}
Let \((E,\Si = \{E_i\}_i)\) be a stratified orbifold and recall the set \(\mc{A}\) of piecewise-\(C^1\) curves in \(E\).  Let \(\h{g} = \lt(g_i\rt)_i\) be a stratified Riemannian (semi-)metric on \(E\) and \(\lt(\ga:[a,b] \to E\rt) \in \mc{A}\).  Since each stratum \(E_i \cc E\) is locally-closed, \(I_i = \ga^{-1}(E_i) \pc [a,b]\) is also locally closed and hence measurable.  Moreover, since \(\ga\) is piecewise-\(C^1\) on the compact interval \([a,b]\), it is Lipschitz continuous on \([a,b]\)  and hence on each \(I_i\) (\wrt\ any Riemannian metric on \(E\)).  It follows from \cite[Lem.\ 3.1.7]{GMT} that \(\dot{\ga}\) lies in the subspace \(\T E_i \pc \T E\), and hence \(g_i\lt(\dot{\ga}\rt)\) is well-defined, almost everywhere on \(I_i\).  Now define \(\h{g}\lt(\dot{\ga}\rt): I \to [0,\0)\) by:
\ew
\h{g}\lt(\dot{\ga}\rt) = g_i\lt(\dot{\ga}\rt) \text{ on } I_i.
\eew
Since each \(g_i\) can be extended to a continuous (semi-)metric \(\overline{g}_i\) on all of \(E\), one has:
\ew
\lt|\h{g}\lt(\dot{\ga}\rt)\rt| \le \max_{i} \sup_{t \in I} \lt|\overline{g}_i\lt(\dot{\ga}(t)\rt)\rt| \text{ almost everywhere}
\eew
and thus \(\h{g}\lt(\dot{\ga}\rt)\) defines a non-negative element of \(L^\0(I)\).  One then defines:
\ew
\ell^{\h{g}}(\ga) = \bigintsss_I \h{g}\lt(\dot{\ga}\rt)^\frac{1}{2} \dd\cal{L}
\eew
where \(\cal{L}\) denotes the Lebesgue measure on \(I\).  \(\lt(\mc{A},\ell^{\h{g}}\rt)\) defines a length structure on \(E\); denote the corresponding metric by \(d^{\h{g}}\).  A similar construction applies to stratified quasi-Finslerian structures \(\h{\mc{L}}\), resulting in a metric \(d^{\h{\mc{L}}}\).
\end{Defn}

Note that any two stratified Riemannian metrics \(\h{g}\) and \(\h{h}\) on \(E\) are locally uniformly Lipschitz equivalent on \(E\), in the sense that for all \(K \C E\) compact, there exists a constant \(C(K)>0\) such that for all \(E_i\):
\ew
\frac{1}{C(K)}g_i \le h_i \le C(K)g_i \text{ on } E_i \cap K.
\eew
(Indeed, by compactness of \(K\), for each \(E_i\), there exists \(C_i(K) >0\) such that \(\frac{1}{C_i(K)}\overline{g}_i \le \overline{h}_i \le C_i(K)\overline{g}_i\) on \(K\); hence the result follows by setting \(C(K) = \max_i C_i(K)\).)  In particular, the metric \(d^{\h{g}}\) induces the usual topology on \(E\).  The analogous result holds for stratified quasi-Finslerian structures.

By contrast, two stratified Riemannian semi-metrics are not, in general, even pointwise Lipschitz equivalent.  However if one fixes a stratified distribution \(\mc{D}\) on \(E\), then any two stratified Riemannian semi-metrics \(\h{g}\) and \(\h{h}\) on \(E\) which are regular \wrt\ \(\mc{D}\) are locally uniformly Lipschitz equivalent.

\begin{Rk}[Refinement of stratification]
Every orbifold Riemannian (semi-)metric \(g\) on a stratified orbifold \((E,\Si = \{E_i\}_i)\) defines a stratified Riemannian (semi-)metric \(\h{g}\) on \(E\) in the obvious way, by setting \(g_i = g|_{E_i}\) (where \(g|_{E_i}\) denotes the tangential component of \(g\) along \(E_i\)).  Then \(\ell^{\h{g}}(\ga) = \ell^g(\ga)\) for all piecewise-\(C^1\) paths in \(E\) and hence \(d^g = d^{\h{g}}\), i.e.\ the (semi-)metrics induced by \(g\) and \(\h{g}\) are the same.

More generally, given a stratified orbifold \((E, \Si = \{E_i\}_i)\), recall that I term a second stratification \(\Si' = \{E'_j\}_j\) of \(E\) a refinement of \(\Si\) if for every \(j\), there exists a (necessarily unique) \(i(j) \in \{0,...,n\}\) such that \(E'_j \cc E_{i(j)}\).  Given a stratified Riemannian (semi-)metric \(\h{g} = \{g_i\}_i\) on \(E\) \wrt\ the stratification \(\Si\), it also defines a stratified Riemannian (semi-)metric \(\h{g}' = \{g'_j\}_j\) \wrt\ \(\Si'\) via \(g'_j = g_{i(j)}|_{E'_j}\).  It is again clear that \(d^{\h{g}} = d^{\h{g}'}\).  The corresponding results for stratified quasi-Finslerian structures are also valid.
\end{Rk}

\begin{A-d}
Let \(E\) be a stratified orbifold and let \(\h{\mc{L}}\) be a stratified quasi-Finslerian structure on \(E\).   Then the length-structure \(\ell^{\h{\mc{L}}}\) has the surprising property that the quantity \(\ell^{\h{\mc{L}}}(\ga)\) does not depend continuously on the piecewise-\(C^1\) curve \(\ga\).  This phenomenon can be observed even on an unstratified manifold; see \cite[Example 2.4.4]{ACiMG}.
\end{A-d}

\section{Statement of main result}\label{MT-section}

The purpose of this section is to present a precise statement of the main theorem of this paper.  I begin by introducing the necessary notation:

\begin{Nt}\label{cvgce-nt}~

\noindent1. \ \parbox[t]{16.2cm}{Let \(E_2\) be a compact orbifold, let \((B,\Si_B = \{B_i\}_i)\) be a stratified orbifold, let \(\pi:E_2 \to B\) be a surjective, weak submersion with path-connected fibres and let \(\Si_2 = \{(E_2)_j\}_j\) be the induced stratification on \(E_2\) (see \cref{ind-strat}).  Let \(\h{g}^\0\) be a stratified Riemannian semi-metric on \(E_2\) which is regular \wrt\ the stratified distribution \(\ker\dd\pi\) and write \(d^\0\) for the semi-metric on \(E_2\) induced by \(\h{g}^\0\).}\\

\noindent2. \ \parbox[t]{16.2cm}{Let \((E_1,\Si_1)\) be another compact, stratified orbifold and, for \(i = 1,2\), let \(S_i(j) \cc E_i\) (\(j = 1,...,N\)) be disjoint, closed, subsets.  Write:
\ew
S_i = \coprod_{j=1}^N S_i(j)
\eew
and suppose there is a stratified orbifold diffeomorphism \(\Ph: E_1 \osr S_1  \to  E_2 \osr S_2\).}\\

\noindent3. \ \parbox[t]{16.2cm}{For each \(j = 1,...,N\), let \(\lt(U^{(r)}_1(j)\rt)_{r \in (0,1]}\) be a family of open neighbourhoods of \(S_1(j) \cc E_1\) such that \(U^{(r)}_1 \cc U^{(s)}_1\) for \(r < s\).  Suppose moreover that for all \(j \ne j' \in \{1,...,N\}\):
\e\label{dj}
\overline{U^{(1)}_1(j)} \cap \overline{U^{(1)}_1(j')} = \es.
\ee
Write \(U^{(r)}_2(j) = \lt. E_2 \m\osr \Ph\lt(E_1 \m\osr U^{(r)}_1(j) \rt) \rt.\) for the corresponding nested open neighbourhoods of \(S_2(j) \cc E_2\), where \(\overline{U^{(1)}_2(j)}\) and \(\overline{U^{(1)}_2(j')}\) are, again, disjoint for distinct \(j\) and \(j'\).  Write:
\ew
U^{(r)}_i = \coprod_{j=1}^N U^{(r)}_i(j), \hs{2mm} i= 1,2.
\eew}\\

\noindent4. \ \parbox[t]{16.2cm}{For each \(j = 1,...,N\), let \(S(j)\) be a set and let \(\fr{f}_{i,j}: U^{(1)}_i(j) \to S(j)\) be surjective maps such that the following diagram commutes:}
\ew
\bcd
\lt. U^{(1)}_1(j) \m\osr S_1(j) \rt. \ar[rr, " \Ph "] \ar[rd, " \fr{f}_{1,j} "'] & & \lt. U^{(1)}_2(j) \m\osr S_2(j) \rt. \ar[dl, " \fr{f}_{2,j} "]\\
& S(j) &
\ecd
\eew
\end{Nt}

Intuitively, one should think of \(S_1\) and \(S_2\) as representing `singular' regions in \(E_1\) and \(E_2\) respectively.  The existence of \(\Ph\) then asserts that the orbifolds \(E_1\) and \(E_2\) are diffeomorphic `away from their singular regions' and condition 4 states that, for each \(j\), the singular regions \(S_1(j)\) and \(S_2(j)\) are `fibred' over a common base space \(S(j)\).  Using \(\Ph\), I shall henceforth identify \(E_1 \osr S_1\) with \(E_2 \osr S_2\) and write \(E^{(r)} = E_1 \osr U^{(r)}_1 \cong E_2 \osr U^{(r)}_2\).  Similarly, I shall write \(\del^{(r)}(j) = \del U_1^{(r)}(j) \cong \del U_2^{(r)}(j)\).

I now state the full version of the main result of this paper:

\begin{Thm}\label{Cvgce-Thm-1-FV}~

\em\noindent 1. \ \parbox[t]{16.2cm}{Fix a stratum \(B_i\) in \(B\), write \(\pi^{-1}(B_i) = \bigcup_{l = 0}^k (E_2)_{j(l)}\) and write \(g_{j(l)}\) for the component of \(\h{g}^\0\) on the stratum \((E_2)_{j(l)}\).  Define a map \(\mc{L}_i:\T B_i \to \bb{R}\) as follows: given \(p\in B_i\) and \(u \in \T_pB_i\), define
\e\label{mcL}
\mc{L}_i(u) = \min_{l = 0}^k ~ \inf_{x \in (E_2)_{j(l)} \cap \pi^{-1}(p)} \lt\{\|u'\|_{g_{j(l)}}~\m|~u' \in \T_x (E_2)_{j(l)} \text{ such that } \dd\pi(u') = u \rt\}.
\ee
Then \(\h{\mc{L}} = \lt\{\mc{L}_i\rt\}_i\) defines a stratified quasi-Finslerian structure on \(B\) and \(\lt(B,\h{\mc{L}}\rt)\) is the free metric space on \(\lt(E_2,\h{g}^\0\rt)\).}\\

\noindent2. \ \parbox[t]{16.2cm}{Now suppose further that \(\lt(\h{g}^\mu\rt)_{\mu \in [1,\0)}\) are stratified Riemannian metrics on \(E_1\) inducing metrics \(d^\mu\), such that the following 4 conditions are satisfied:

(i) \ For all \(r \in (0,1]\):
\ew
\h{g}^\mu \to \h{g}^\0 \text{ uniformly as } \mu \to \0 \text{ on the space } E^{(r)}
\eew
and there exist constants \(\La_\mu(r) \ge 0\) such that:
\e\label{La-cvgce}
\lim_{\mu \to \0} \La_\mu(r) = 1 \et \h{g}^\mu \ge \La_\mu(r)^2 \h{g}^\0 \text{ on } E^{(r)} \text{ for all } \mu \in [1,\0);
\ee

(ii)
\ew
\limsup_{r \to 0} ~ \limsup_{\mu \to \0} ~\max_{j \in \{1,...,N\}} ~ \sup_{p \in S(j)} ~ \diam_{d^\mu}\lt[ \fr{f}_{1,j}^{-1}(\{p\}) \cap U^{(r)}_1(j) \rt] = 0;
\eew}

\parbox[t]{16.2cm}{(iii)
\ew
\limsup_{r \to 0} ~ \max_{j \in \{1,...,N\}} ~ \sup_{p \in S} ~ \diam_{d^\0}\lt[ \fr{f}_{2,j}^{-1}(\{p\}) \cap U^{(r)}_2(j) \rt] = 0;
\eew

(iv)
\ew
\limsup_{r \to 0} ~ \limsup_{\mu \to \0} ~\max_{j \in \{1,...,N\}} ~ \sup_{\del^{(r)}(j)} \lt| d^\mu - d^\0 \rt| = 0.
\eew

Then:
\ew
\lt(E_1,d^\mu\rt) \to \lt(B,\h{\mc{L}}\rt) \as \mu \to \0,
\eew
in the Gromov--Hausdorff sense.}
\end{Thm}

\noindent(By `\(\h{g}^\mu \to \h{g}^\0\) uniformly on \(E^{(r)}\)', it is meant that for each fixed reference stratified Riemannian metric \(\h{h}\) on \(E_1\), one has:
\ew
\lt\|g^\mu_i - g^\0_i\rt\|_{h_i} \to 0 \as \mu \to \0 \text{ on } E_i \cap E^{(r)}.
\eew
Since any two stratified Riemannian metrics on a compact orbifold are uniformly Lipschitz equivalent, this definition is independent of the choice of \(\h{h}\).)

\begin{Rk}
Note that if the bilinear form \(\h{g}^\0\) is given on each stratum \((E_2)_j\) lying over \(B_i\) as \(\pi^*h_i\), where \(\h{h} = \{h_i\}_i\) is a stratified Riemannian metric on \(B\) (write \(\h{g}^\0 = \pi^*\h{h})\), then \(\h{\mc{L}} = \h{h}\), and \(\lt(B,d^{\h{h}}\rt)\) is clearly the free metric space on \(\lt(E_2,d^\0\rt)\).  More generally, if \(\h{g}^\0 = \pi^*\h{h}\) on some proper subset \(U\) of \(E_2\), then once again it is clear that \(\h{\mc{L}} = \h{h}\) over \(U\), since the value of \(\h{\mc{L}}\) at a point \(b \in B\) only depends on the values of \(\h{g}^\0\) on the fibre over \(b\).  However due to the global definition of the metrics \(d_B\) and \(d^{\h{\mc{L}}}\), the assumption \(\h{g}^\0 = \pi^*\h{h}\) on \(U\) provides no simplification and the proof that \(d_B = d^{\h{\mc{L}}}\) -- far from being trivial -- assumes its general form in this case.
\end{Rk}

Intuitively, condition (i) states that the metrics \(\h{g}^\mu\) on \(E_1\) converge locally uniformly away from the singular region \(S_1\) to a stratified Riemannian semi-metric \(\h{g}^\0\), which extends to some given compactification \(E_2\) of \(E_1 \osr S_1\); conditions (ii) and (iii) state that the fibres of the maps \(f_{i,j}: U_i^{(r)}(j) \to S(j)\) are `small' \wrt\ \(\h{g}^\mu\) and \(\h{g}^\0\) respectively, provided that \(\mu\) is sufficiently large and \(r\) is sufficiently small and, finally, condition (iv) states that the two metrics \(d^\mu\) and \(d^\0\) approximately agree near the singular regions \(S_1\) and \(S_2\).  Further intuition can be gained by considering the case where the map \(\Ph: E_1 \osr S_1 \to E_2 \osr S_2\) admits a non-smooth extension to a map \(\overline{\Ph}:E_1 \to E_2\).  In this case, by considering the composite \(\pi \circ \overline{\Ph}: E_1 \to B\), one can regard \(E_1\) as fibred over the space \(B\) and \tref{Cvgce-Thm-1-FV} can then be regarded as a `collapsing' theorem for this fibration which states, informally, that if the diameter of the fibres of \(E_1 \to B\) (\wrt\ the metrics \(\h{g}^\mu\)) tend to zero away from some singular set \(S_1\) and if the limiting size of the region \(S_1\) is `not too large', then the orbifold \(E_1\) collapses to the orbifold \(B\) in the limit as \(\mu \to \0\).  The use of \tref{Cvgce-Thm-1-FV} in \cite{URftHFoG23F} is an example of such an application.

The proof of \tref{Cvgce-Thm-1-FV} occupies the rest of this paper.  I begin by explaining why the hypotheses of \tref{Cvgce-Thm-1-FV} are necessary.

Firstly, to understand why it is necessary to use quasi-Finslerian structures, rather than the more usual Finslerian or Riemannian structures to describe the free metric space on \(\lt(E_2,\h{g}^\0\rt)\), consider the following simple (unstratified) example, where \(\la > 2\) is a constant:
\caw
\pi : E = \bb{T}^2_{\th^1,\th^2} \x \bb{T}^1_{\al} \oto{proj} \bb{T}^2_{\th^1,\th^2} = B\\
g = \lt(1 + \la\cos^2\al\rt)\lt(\dd \th^1\rt)^{\ts2} + \lt(1 + \la\sin^2\al\rt)\lt(\dd \th^2\rt)^{\ts2}
\caaw
Write \(\del_i = \frac{\del}{\del \th^i}\) (\(i=1,2\)), let \(p = \lt(\th^1,\th^2, \al\rt) \in E\) and for each \(a,b \in \bb{R}\), let \(u(a,b)\) be any vector in \(\T_p\bb{T}^3\) such that \(\dd\pi|_p(u(a,b)) = a\del_1 + b\del_2\).  Then:
\ew
\|u(a,b)\|_g = \sqrt{a^2 + (\la+1)b^2 + \la\lt(a^2-b^2\rt)\cos^2\al} = \sqrt{(\la+1)a^2 + b^2 + \la\lt(b^2 - a^2\rt)\sin^2\al}
\eew
and thus:
\ew
\mc{L}(a\del_1 + b\del_2) =
\begin{dcases*}
\sqrt{a^2 + (\la+1)b^2} &if \(|a| \ge |b|\)\\
\sqrt{(\la+1)a^2 + b^2} &if \(|a| \le |b|\).
\end{dcases*}
\eew
Whilst this function is continuous, it is not differentiable along \(a=b\).  Moreover:
\ew
\mc{L}(\del_1 + \del_2) = \sqrt{\la+2} > 2 = \mc{L}(\del_1) + \mc{L}(\del_2)
\eew
and thus \(\mc{L}\) does not satisfy the triangle inequality.  This is the motivation behind the definition of stratified quasi-Finslerian structures in \sref{Strat-data}.

Secondly, to understand why the existence of \(\La_\mu(r) \to 1\) in condition (i) is necessary for the second conclusion of \tref{Cvgce-Thm-1-FV} to be valid, take \(E_1 = E_2 = \bb{T}^2_{\th^1,\th^2}\) with the trivial (1-stratum) stratifications, let \(U^{(r)}_i = S_i = \es\) for \(i = 1,2\), let \(\h{g}^\0 = \lt(\dd \th^1\rt)^{\ts2}\) and let:
\ew
\h{g}^\mu = \lt(1+\mu^{-1}\rt)\lt(\dd \th^1\rt)^{\ts2} - 2\mu^{-1}\dd \th^1 \s \dd \th^2 + \mu^{-2} \lt(\dd \th^2\rt)^{\ts2}.
\eew
Since \(U^{(r)}_i = S_i = \es\), conditions (ii)--(iv) in \tref{Cvgce-Thm-1-FV} are automatically satisfied.  Moreover, since \(\h{g}^\mu \to \h{g}^\0\) uniformly as \(\mu\to\0\), condition (i) is also satisfied, expect for the existence of suitable \(\La_\mu\).  However \(d^\mu \to 0\) uniformly as \(\mu \to \0\).  Indeed, for each \(a \in [0,1]\), consider the path:
\ew
\bcd[row sep = 0pt]
\ga: [0,a] \ar[r]& \bb{T}^2\\
s \ar[r, maps to]& \lt(s,\mu \cdot s\rt)
\ecd
\eew
Then one may calculate that \(g^\mu\lt(\dot{\ga}\rt) = \mu^{-1}\) and thus:
\ew
d^\mu\lt[(0,0), (a, \mu \cdot a)\rt] \le a\mu^{-\frac{1}{2}} \le \mu^{-\frac{1}{2}}.
\eew
Likewise, by considering a vertical path between \((a,b)\) and \((a,\mu \1 b)\) one sees that \(d^\mu\lt[ (a,b),(a,\mu \1 a)\rt] \le \mu^{-1}\) for any \(b \in [0,1]\) and thus for all \((a,b) \in \bb{T}^2\):
\ew
d^\mu\lt[ (0,0), (a,b)\rt] \le \mu^{-1} + \mu^{-\frac{1}{2}} \to 0 \as \mu \to \0.
\eew
By \cite[Example 7.4.4]{ACiMG} it follows that \(\lt(\bb{T}^2, \h{g}^\mu\rt)\) converges to the one-point space in the \GH\ sense, as \(\mu \to \0\).  However, the free metric space on \(\lt(\bb{T}^2, \h{g}^\0\rt)\) is \(\lt(\bb{T}^1_{\th^1}, \lt(\dd\th^1\rt)^{\ts2}\rt)\).  This shows the necessity of the existence of the \(\La_\mu \to 1\).  An analogous phenomenon was observed (albeit from a very different perspective) in \cite{SCafSS}.

(To see why no such \(\La_\mu \to 1\) can exist, suppose that the inequalities \(g^\mu \ge \La_\mu^2 g^\0\) held for all \(\mu \in [1,\0)\).  Then the bilinear form \(g^\mu - \La_\mu^2g^\0\) would be non-negative definite and hence would have non-negative determinant, i.e.:
\ew
\mu^{-2}\lt(1+\mu^{-1} - \La_\mu^2\rt) - \mu^{-2} \ge 0.
\eew
This rearranges to \(\La_\mu^2 \le \mu^{-1}\) and hence would force \(\La_\mu\to0\) as \(\mu\to\0\).)

\section{Proof of \tref{Cvgce-Thm-1-FV}: Part 1}\label{Gen-Coll}

The purpose of this section is to prove the first part of \tref{Cvgce-Thm-1-FV}.  The reader should note that the orbifold \(E_1\), and also the sets \(S_i(j)\), \(U_i^{(r)}(j)\) and their associated data, play no role in this section.  Thus in this section, for simplicity of notation, I denote \(E_2\) by \(E\), \(\Si_2\) by \(\Si\) and \(\h{g}^\0\) by \(\h{g}\).  The reader should also note that the results of this section remain valid when \(E\) is non-compact, provided that the map \(\pi\) is proper.

The first task is to verify that \(\h{\mc{L}}\) is a well-defined stratified quasi-Finslerian structure on \(B\):
\begin{Prop}\label{strat-well-def}
For each stratum \(B_i\) in \(B\), recall the definition:
\ew
\mc{L}_i(u) = \min_{l = 0}^k ~ \inf_{x \in E_{j(l)} \cap \pi^{-1}(p)} \lt\{\|u'\|_{g_{j(l)}}~\m|~u' \in \T_x E_{j(l)} \text{ such that } \dd\pi(u') = u \rt\},
\eew
for \(u \in \T_pB_i\), where \(\pi^{-1}(B_i) = \bigcup_{l = 0}^k E_{j(l)}\) and \(g_{j(l)}\) is the component of \(\h{g}\) on the stratum \(E_{j(l)}\).  Then \(\h{\mc{L}} = \lt\{\mc{L}_i\rt\}_i\) defines a stratified quasi-Finslerian structure on \(B\).
\end{Prop}

\begin{proof}
Firstly, note that \(\mc{L}_i\) is well-defined.  Indeed \(\pi|_{E_{j(l)}}: E_{j(l)} \to B_i\) is a surjective submersion for each \(l \in \{0,...,k\}\) and so for each \(x \in E_{j(l)} \cap \pi^{-1}(p)\) there exist vectors \(u' \in \T_x E_{j(l)}\) satisfying \(\dd\pi(u') = u\).   Moreover \(\|u'\|_{g_{j(l)}}\) is independent of the choice of \(u'\), since any two such choices of \(u'\) differ by an element of \(\ker(\dd\pi)|_x \cap \T_x E_{j(l)}\) which is in turn precisely the kernel of \(g_{j(l)}\), since \(\h{g}\) is regular \wrt\ \(\ker\dd\pi\).

The proof now breaks into two cases of increasing generality:\\

\noindent{\bf Case 1: \(B\) is a manifold.}  Let \(p \in B\), \(u \in \T_p B\) and \(x \in E_i \cap \pi^{-1}(p)\) for some stratum \(E_i\) of \(E\).  It is beneficial to have a preferred choice of preimage of \(u\) under \(\dd\pi: \T_x E \to \T_p B\).  To this end, let \(\mc{C}\) be the (stratified) orthocomplement to \(\mc{D}\) \wrt\ some Riemannian metric on \(E\).  Then (cf.\ \pref{EIFT}):
\ew
\dd\pi_x: \mc{C}_x \to \T_p B
\eew
is an isomorphism for all \(x \in E_i\); let \(u_x\) denote the preimage of \(u\) under this isomorphism.  Given any \(u' \in T_x E_i\) such that \(\dd\pi(u') = u\), the difference \(u' - u_x\) lies in \(\ker\dd\pi\) and, since \(\h{g}\) is regular \wrt\ \(\ker\dd\pi\), it follows that \(g_i(u') = \ol{g}_i(u_x)\) (where \(\ol{g}_i\) is the continuous extension of \(g_i\) to all of \(E\)) and hence:
\e\label{mcL-2}
\mc{L}(u) = \min_i ~ \underbrace{\inf \lt\{\|u_x\|_{\ol{g}_i}~\m|~x \in \pi^{-1}(p) \cap E_i\rt\}}_{= \mc{L}_i}.
\ee

Next, note that given quasi-Finlserian structures \(\mc{L}_1\), \(\mc{L}_2\) on \(B\), their pointwise minimum \(\mc{L} = \min(\mc{L}_1,\mc{L}_2)\) is also a quasi-Finslerian structure.  Thus, it suffices to prove that each \(\mc{L}_i\) defines a quasi-Finlserian structure on \(B\).

It is clear from \eref{mcL-2} that \(\mc{L}_i\) is non-negative and satisfies:
\ew
\mc{L}_i(\la\cdot u) = |\la|\mc{L}_i(u)
\eew
for all \(u \in \T B\) and \(\la \in \bb{R}\).   To see that \(\mc{L}_i\) is positive-definite, let \(u \in \T_p B\) satisfy \(\mc{L}_i(u) = 0\).   Choose a sequence \(x_n \in \pi^{-1}(p) \cap E_i\) such that \(\|u_{x_n}\|_{\ol{g}_i} \to 0 \) as \(n \to \0\).   Since \(\pi\) is proper, \(\pi^{-1}(p)\) is compact and thus \(x_n\) converges subsequentially to some \(x \in \pi^{-1}(p)\); passing to a subsequence, one may assume \wlg\ that \(x_n \to x\).  Then \(\|u_{x_n}\|_{\ol{g}_i} \to \|u_x\|_{\ol{g}_i}\) as \(n \to \0\), hence \(\|u_x\|_{\ol{g}_i} = 0\) and whence \(u_x = 0\), since \(\ol{g}_i\) is non-degenerate on \(\mc{C}\), which in turn follows from the regularity of \(\h{g}\) \wrt\ \(\mc{D}\).  Thus:
\ew
u = \dd\pi(u_x) = 0,
\eew
as required.  Thus to prove that \(\mc{L}_i\) is a quasi-Finlserian structure, it suffices to prove continuity.

To this end, choose \(b \in B\), \(u \in \T_bB\) and pick a sequence \(u_n \in \T_{b_n}B\) tending to \(u\) as \(n \to \0\) (in particular, \(b_n \to b\) as \(n \to \0\)).  Pick a sequence \(x_m \in \pi^{-1}(b) \cap E_i\) such that:
\ew
\mc{L}_i(u) = \lim_{m\to\0} \|u_{x_m}\|_{\ol{g}_i}.
\eew
As above, by properness of \(\pi\), \(x_m\) converges subsequentially to some \(x \in \pi^{-1}(b)\); by passing to a subsequence, \wlg\ \(x_m \to x\) as \(m \to \0\).  By \pref{loc-triv}, one can choose a neighbourhood \(U_b\) of \(b\) satisfying \(\pi^{-1}(U_b) \cong U_b \x \pi^{-1}(b)\).  \Wlg\ assume that all \(b_n\) lie in \(U_b\).  Note that both \((u_n)_{(b_n,x_n)}\) and \(u_{(b,x_n)}\) converge to \(u_{(b,x)} \in \mc{C}\) as \(n\to\0\).  Thus:
\e\label{pre-limsup}
\|(u_n)_{(b_n,x_n)}\|_{\ol{g}_i},  \|u_{(b,x_n)}\|_{\ol{g}_i} \to \|u_{(b,x)}\|_{\overline{g}_i} \as n \to \0
\ee
(note that \((b,x_n) \in E_i\) implies that \((b_n,x_n) \in E_i\); see \rref{strat-rk}).  Now clearly:
\ew
\|(u_n)_{(b_n,x_n)}\|_{\ol{g}_i} \ge \mc{L}_i(u_n)
\eew
for each \(n\).  Thus \eref{pre-limsup} implies that:
\ew
\limsup_{n \to \0} \mc{L}_i(u_n) \le \lim_{n \to \0}\|(u_n)_{(b_n,x_n)}\|_{\ol{g}_i} = \lim_{n \to \0} \|u_{(b,x_n)}\|_{\ol{g}_i} = \mc{L}_i(u).
\eew

Thus suppose, for the sake of contradiction, that \(\liminf_{n \to \0} \mc{L}_i(u_n) \le \mc{L}_i(u) - \eta\) for some \(\eta>0\).   Then \(\mc{L}_i(u_n) \le \mc{L}_i(u) - \eta\) for infinitely many \(n\) and thus, \wlg, for all \(n\), by passing to a subsequence if necessary.  For each \(n\), choose \(y_n \in \pi^{-1}(b) \cap E_i\) such that:
\e\label{eta/3}
\lt\|(u_n)_{(b_n,y_n)}\rt\|_{\ol{g}_i} \le \mc{L}_i(u_n) + \frac{\eta}{3}.
\ee
Using compactness of \(\pi^{-1}(b)\) again, by passing to a subsequence if necessary \(y_n \to y \in \pi^{-1}(b)\).  Thus as above:
\ew
\|(u_n)_{(b_n,y_n)}\|_{\ol{g}_i},  \|u_{(b,y_n)}\|_{\ol{g}_i} \to \|u_{(b,y)}\|_{\ol{g}_i} \as n \to \0.
\eew

Choose \(n\) sufficiently large so that \(\lt|\|(u_n)_{(b_n,y_n)}\|_{\ol{g}_i} - \|u_{(b,y_n)}\|_{\ol{g}_i}\rt| \le \frac{\eta}{3}\). Then:
\ew
\mc{L}_i(u) &\le \|u_{(b,y_n)}\|_{\ol{g}_i}\\
&\le \|(u_n)_{(b_n,y_n)}\|_{\ol{g}_i} + \frac{\eta}{3}\\
&\le \mc{L}_i(u_n) + \frac{2\eta}{3} \hs{5mm} \text{(by \eref{eta/3})}.
\eew
However by assumption \(\mc{L}_i(u_n) \le \mc{L}_i(u) - \eta\) and thus \(\mc{L}_i(u) \le \mc{L}_i(u) - \frac{\eta}{3}\), yielding a contradiction.  Thus \(\liminf_{n \to \0} \mc{L}_i(u_n) \ge \mc{L}_i(u)\) and so \(\mc{L}_i(u_n) \to \mc{L}_i(u)\) as \(n \to \0\), proving that \(\mc{L}_i\) is continuous, as required.\\

\noindent{\bf Case 2: General Case.}  Using Case 1, for each stratum \(B_i\) of \(B\), the function \(\mc{L}_i\) is a quasi-Finslerian structure on \(B_i\).  Thus, to prove that \(\h{\mc{L}} = \{\mc{L}_i\}_i\) is a stratified quasi-Finslerian structure on \(B\) -- and hence to complete the proof of \pref{strat-well-def} -- it suffices to prove that given a Riemannian metric \(h\) on \(B\), each \(\mc{L}_i\) is Lipschitz equivalent to \(h\) up to the boundary of \(B_i\).   To this end, fix a stratum \(E_j\) of \(E\) and recall the extension \(\overline{g}_j\) of \(g_j\) to all of \(E\).  Since both the Riemannian semi-metrics \(\pi^*h\) and \(\overline{g}_j\) vanish on \(\mc{D} = \ker\dd\pi\) and are positive-definite on \(\mc{C}\), it follows that there is a continuous map \(D:E \to (0,\0)\) such that:
\ew
\frac{1}{D}\pi^*h \le \overline{g}_j \le D\pi^*h \text{ on all of } E.
\eew
Since \(\pi\) is proper, one may define a continuous map \(C:B \to (0,\0)\) such that for all \(e \in E\):
\ew
D(e) \le C(\pi(e)).
\eew
Then it follows immediately from \eref{mcL-2} that:
\ew
\frac{1}{C}\|-\|_h \le \mc{L}_i \le C\|-\|_h \text{ on } B_i.
\eew
This completes the proof.

\end{proof}

I now prove the first part of \tref{Cvgce-Thm-1-FV}:

\begin{Prop}\label{FMS-Thm}
Let \(\h{\mc{L}}\) be as in \pref{strat-well-def}.   Then \(\lt(B,d^{\h{\mc{L}}}\rt)\) is the free metric space on \(\lt(E,d^{\h{g}}\rt)\).
\end{Prop}

The proof proceeds by a series of lemmas.

\begin{Lem}[`{\it A priori} bound']\label{DCT-bd}
Let \(\pi: E \to B\) be as in \ntref{cvgce-nt} and let \(\ga\) be a piecewise-\(C^1\) path in \(B\).  There exists a constant \(C>0\), depending only on \(\ga\), such that for all piecewise-\(C^1\) lifts \(\tld{\ga}\) of \(\ga\) along \(\pi\):
\ew
\h{g}\lt(\dot{\tld{\ga}}\rt) \le C \text{ almost everywhere}.
\eew
\end{Lem}

\begin{proof}
The argument is similar to the proof of the general case in \pref{strat-well-def}.  By considering each \(C^1\) portion of \(\ga\) separately, \wlg\ assume that \(\ga\) is \(C^1\).  Let \(h\) be a Riemannian metric on \(B\) and recall that there is a continuous map \(c:B \to (0,\0)\) such that, for each stratum \(E_j\) of \(E\), writing \(\overline{g}_j\) for the continuous extension of \(g_j\) to all of \(E\):
\ew
\overline{g}_j|_e \le c(\pi(e))\pi^*h|_e \text{ for all } e \in E.
\eew
Thus for all \(t\) in the domain of definition of \(\ga\):
\ew
\overline{g}_j\lt(\dot{\tld{\ga}}\rt)|_{\tld{\ga}(t)} \le c(\ga(t)) h\lt(\dot{\ga}\rt)|_{\ga(t)}.
\eew
Since the domain of definition of \(\ga\) is compact, the right-hand side may be bounded uniformly above by some \(C > 0\) depending only on \(\ga\) and the arbitrary \rmm\ \(h\).  The result follows.

\end{proof}

\begin{Lem}[`Piecewise-\(C^1\) Lifts with Specified Endpoints']\label{lifting-lem}
Let \(\ga:[0,1] \to B\) be a \(C^1\) path in \(B\), let \(b_i = \ga(i)\) and let \(e_i \in \pi^{-1}(b_i)\) for \(i = 0,1\) respectively.  Then there exists a piecewise-\(C^1\) \(\tld{\ga}: [0,1] \to E\) such that \(\pi\lt(\tld{\ga}\rt) = \ga\) and \(\tld{\ga}(i) = e_i\) for \(i = 0,1\) respectively.
\end{Lem}

\begin{Rk}
Note that the ability to lift paths along \(\pi\) is a non-trivial result since, as explained in \exref{FibrEx}, \(\pi\) need not even be a Serre fibration.  In particular, \lref{lifting-lem} does not appear to follow from any known result in the literature.
\end{Rk}

\begin{proof}
Firstly, I claim that there exists a piecewise-\(C^1\) lift \(\h{\ga}\) of \(\ga\) along \(\pi\) satisfying \(\h{\ga}(0) = e_0\).  Indeed, by applying \pref{EIFT}, one can choose a chart \(\Xi_{e_0} = (U_{e_0}, \Ga_{e_0}, \bb{A}_1 \x \bb{A}_2, \ch_{e_0})\) about \(e_0\) in \(E\), a chart \(\Xi_{b_0} = (U_{b_0}, \Ga_{b_0}, \bb{A}_2, \ch_{b_0})\) about \(b_0\) in \(B\) and a homomorphism \(\ka_\pi: \Ga_{e_0} \to \Ga_{b_0}\) such that \(\pi\) may be locally lifted as:
\ew
\bcd[row sep = -2mm, column sep = 20mm]
\underbrace{\bb{A}_1 \x \bb{A}_2}_{\circlearrowleft} \ar[r, " \tld{\pi} = proj "] & \underbrace{\bb{A}_2}_{\circlearrowleft}\\
\Ga_{e_0} \ar[r, " \ka_\pi "'] & \Ga_{b_0}
\ecd
\eew
and where \(\Ga_{e_0}\) acts on \(\bb{A}_2\) via \(\ka_\pi\).  Let \(\overline{\ga}:[0,\ep) \to \bb{A}_2\) for some \(\ep > 0\) be a local representation of \(\ga\) in the chart \(\Xi_{b_0}\) (note that the equivariance of this local representation is trivial, since \([0,\ep)\) is a manifold and so its orbifold group about each point vanishes).  Lift \(\overline{\ga}\) to a map \([0,\ep) \to \bb{A}_1 \x \bb{A}_2\) as:
\ew
0 \x \overline{\ga} : [0,\ep) \to \bb{A}_1 \x \bb{A}_2.
\eew
Under the projection \(\bb{A}_1 \x \bb{A}_2 \oto{proj} \lqt{\bb{A}_1 \x \bb{A}_2}{\Ga_{e_0}}\), the map \(0 \x \overline{\ga}\) defines a \(C^1\) lift \(\h{\ga}\) of \(\ga\) along \(\pi\) on the interval \([0,\ep)\).  To obtain a piecewise-\(C^1\) lift of the path \(\ga\) over the whole interval \([0,1]\), one then proceeds inductively, repeating the above process starting from some point \(\ep' \in [0, \ep)\); the inductive process can be made to terminate in finite time by compactness of \([0,1]\) and properness of the map \(\pi\).

Next, I claim that one may `improve' the lift from \(\h{\ga}\) to a lift \(\tld{\ga}\) such that \(\tld{\ga}(b_i) = e_i\) for both \(i = 0 \) and \(i = 1\).  To this end, define \(e_1' = \h{\ga}(1)\).  Since \(\pi\) has path connected fibres, one may choose a piecewise-\(C^1\) path \(\si:[0,1] \to \pi^{-1}(b_1)\) such that \(\si(0) = e_1'\) and \(\si(1) = e_1\).  The task is to deform the endpoint of \(\h{\ga}\) along the path \(\si\).

The argument is very similar to the construction of \(\h{\ga}\).  Initially, one chooses charts \(\Xi_{e_1'} = (U_{e_i'}, \Ga_{e_i'}, \bb{B}_1 \x \bb{B}_2, \ch_{e_1'})\) about \(e_1'\) in \(E\) and \(\Xi_{b_1} = (U_{b_1}, \Ga_{b_1}, \bb{B}_2, \ch_{b_1})\) about \(b_1\) in \(B\) and a homomorphism \(\ka_\pi: \Ga_{e_i'} \to \Ga_{b_1}\) such that \(\pi\) can be written in the local form:
\ew
\bcd[row sep = -2mm, column sep = 20mm]
\underbrace{\bb{B}_1 \x \bb{B}_2}_{\circlearrowleft} \ar[r, " \tld{\pi} = proj "] & \underbrace{\bb{B}_2}_{\circlearrowleft}\\
\Ga_{e_1'} \ar[r, " \ka_\pi "'] & \Ga_{b_1}
\ecd
\eew
Since \(\h{\ga}\) is continuous, there exists \(\ep \in (0,1)\) such that \(\h{\ga}((\ep,1]) \pc U_{e_1'}\).  Then choose a local representation \(\overline{\h{\ga}} = \lt(\overline{\h{\ga}}_1,\overline{\h{\ga}}_2\rt): (\ep, 1] \to \bb{B}_1 \x \bb{B}_2\) (so that \(\overline{\h{\ga}}_1(1) = 0\)).  Now choose \(t >0\) such that \(\si(t) \in U_{e_1'}\) and lift the point \(\si(t)\) to some preimage \(\lt(\overline{\si(t)},0\rt)\) under the projection map \(\bb{B}_1 \x \bb{B}_2 \oto{proj} \lqt{\bb{B}_1 \x \bb{B}_2}{\Ga_{e_1'}}\).  By altering the function \(\overline{\h{\ga}}_1:(\ep,1] \to \bb{B}_1\) on some compact subset of \((\ep,1]\), one can ensure that \(\overline{\h{\ga}}_1(1) = \overline{\si(t)}\).  Denote the new local representation \(\lt(\overline{\h{\ga}}_1,\overline{\h{\ga}}_2\rt)\) by \(\overline{\h{\ga}}'\).  Projecting \(\overline{\h{\ga}}'\) under the map:
\ew
\bb{B}_1 \x \bb{B}_2 \oto{proj} \lqt{\bb{B}_1 \x \bb{B}_2}{\Ga_{e_1'}}
\eew
yields a new lift \(\h{\ga}'\) covering \(\ga\) with the property that \(\h{\ga}'(1) = \si(t)\).  Now iterate this argument, noting again that the process can be made to terminate in finite time by compactness of the domain of definition of \(\si\).

\end{proof}

\begin{Lem}[`Convergence in Measure']\label{DCT-seq-constr}
Let \(\ga: [0,1] \to B\) be a piecewise-\(C^1\) path.  Then there exists a sequence of piecewise-\(C^1\) lifts \(\tld{\ga}_n\) of \(\ga\) along \(\pi\) such that:
\ew
\h{g}\lt(\dot{\tld{\ga}}_n\rt)^\frac{1}{2} \to \h{\mc{L}}\lt(\dot{\ga}\rt)^\frac{1}{2} ~ \text{in measure as } n \to \0.
\eew
Moreover, the endpoints of the lift, viz.\ \(\tld{\ga}_n(0)\) and \(\tld{\ga}_n(1)\), may be chosen to be any points in \(\pi^{-1}(\ga(0))\) and \(\pi^{-1}(\ga(1))\) respectively.
\end{Lem}

\begin{proof}
Firstly, note that it suffices to lift each \(C^1\) portion of \(\ga\) separately to a piecewise-\(C^1\) curve in \(E\) and then use the freedom in specifying the endpoints of these separate lifts to ensure that the combined lift of \(\ga\) is continuous over the non-differentiable points of \(\ga\).  Thus \wlg\ one may assume that \(\ga\) is everywhere \(C^1\).

Write \(I = [0,1]\).  For each stratum \(B_i\) of \(B\), recall the measurable subset:
\ew
I_i = \ga^{-1}(B_i) \cc I.
\eew
Write \(\cal{L}\) for the Lebesgue measure on \(I\) and define:
\ew
\dot{I}_i = \lt\{x \in I_i \cap (0,1) ~\m|~ \lim_{r \to 0} \frac{\cal{L}\lt[I_i \cap (x - r,x + r)\rt]}{\cal{L}\lt[(x-r,x+r)\rt]} = 1\rt\} \cc I_i \cap (0,1).
\eew
By \cite[Cor.\ 2.9, p.\ 20]{GMT:ABG}, \(\cal{L}\lt(I_i \m\osr \dot{I}_i\rt) = 0\).  Define:
\ew
\tld{I}_i = \lt\{ x \in \dot{I_i} ~\m|~ \dot{\ga} \in T B_i \rt\}.
\eew
Then by \cite[Lem.\ 3.1.7, p.\ 217]{GMT}, \(\cal{L}\lt(I_i \m\osr \tld{I}_i\rt) = 0\).  For notational convenience later in the proof, define \(\tld{I} = \bigcup_{i} \tld{I}_i\).

Fix \(n \ge 1\).  For each \(i\) and each \(x \in \tld{I}_i\), choose \(j = j(x,n)\) and \(y = y(x,n) \in E_j \cc \pi^{-1}(B_i)\) such that for all lifts \(u\) of \(\dot{\ga}(x)\) to \(\T_y E_j\) along \(\dd\pi\):
\ew
g_j(u)^\frac{1}{2} - \mc{L}_i\lt(\dot{\ga}(x)\rt)^\frac{1}{2} < \frac{1}{n}.
\eew
(Note that the left-hand side is automatically non-negative, by definition of \(\mc{L}_i\).)  Choose a chart \(\Xi_x = (U_x,\Ga_x, \tld{U}_x, \ch_x)\) about \(\ga(x) \in B\) which is regular for the submanifold \(B_i\), with regular subspace \(\bb{I}\) (see \dref{suborb}):
\ew
\tld{U}_x = \bb{A}_1 \x \bb{I} \to \lqt{\bb{A}_1}{\Ga_x} \x \bb{I} \underbrace{\cong}_{\ch_x} U_x.
\eew
(Here, \(\bb{A}_1\) and \(\bb{I}\) are finite-dimensional real vector spaces and \(\Ga_x\) acts on \(\bb{A}_1\).)  By \pref{EIFT}, shrinking \(U_x\) if necessary one can choose a chart \(\Xi_y = \lt(U_y, \Ga_y, \bb{A}_2 \x \bb{A}_1 \x \bb{I}, \ch_y\rt)\) and a homomorphism \(\ka_\pi: \Ga_y \to \Ga_x\) such that \(\pi\) may be expressed in the charts \(\Xi_y\) and \(\Xi_x\) as the projection map:
\e\label{loc-rep-4-pi}
\bcd[row sep = -2mm, column sep = 20mm]
\underbrace{\bb{A}_2 \x \bb{A}_1}_{\circlearrowleft} \x \bb{I} \ar[r, " \tld{\pi} = proj "] & \underbrace{\bb{A}_1}_{\circlearrowleft} \x \bb{I}\\
\Ga_y \hs{3mm} \ar[r, " \ka_\pi "'] & \Ga_x \hs{3mm}
\ecd
\ee

Since \(x \in \tld{I}_i \cc (0,1)\), and \(\ga\) is continuous, one can choose \(\eta(x,n) > 0\) such that \((x - \eta(x,n), x + \eta(x,n)) \cc (0,1)\) and:
\ew
\ga\lt[ (x - \eta(x,n), x + \eta(x,n)) \rt] \pc U_x \cong \lt(\lqt{\bb{A}_1}{\Ga_x}\rt) \x \bb{I}.
\eew
Since \(x \in \tld{I}_i \pc \dot{I}_i\), one has:
\ew
\lim_{r \to 0} \frac{\cal{L}\lt[I_i \cap (x - r,x + r)\rt]}{\cal{L}\lt[(x-r,x+r)\rt]} = \lim_{r \to 0} \frac{\cal{L}\lt[\tld{I}_i \cap (x - r,x + r)\rt]}{\cal{L}\lt[(x-r,x+r)\rt]} = 1
\eew
(where the first equality follows from the fact that \(\cal{L}\lt(I_i \m\osr \tld{I}_i\rt) = 0\)) and so, by reducing \(\eta(x,n) > 0\) if necessary, one may assume that for all \(0 < r \le \eta(x,n)\):
\e\label{density-close-to-1}
\cal{L}\lt[\tld{I}_i \cap (x - r,x + r)\rt] \ge \frac{n-1}{n}\cal{L}\lt[(x-r,x+r)\rt] = 2r\frac{n-1}{n}.
\ee

Since \(\ga\) defines a smooth map from the manifold \(I\) to the orbifold \(B\), by reducing \(\eta(x,n)>0\) still further if necessary, one may assume that \(\ga\) has a local lift \(\overline{\ga}\) \wrt\ the coordinate charts \((x - \eta(x,n), x + \eta(x,n))\) and \(\Xi_x\):
\ew
\bcd
& \bb{A}_1 \x \bb{I} \ar[d]\\
(x - \eta(x,n), x + \eta(x,n)) \ar[ur, " \mLa{\overline{\ga}} "] \ar[r, " \mLa{\ga} "] & \lqt{\bb{A}_1}{\Ga_x} \x \bb{I} \cong U_x
\ecd
\eew
(Note that the equivariance of the lift \(\overline{\ga}\) is vacuous, since \(I\) is a manifold and so has trivial orbifold groups about every point.)  Using the local representation of \(\pi\) given in \eref{loc-rep-4-pi}, one can lift \(\overline{\ga}\) to the map \(0 \x \overline{\ga}\) as below:
\ew
\bcd
& \bb{A}_2 \x \bb{A}_1 \x \bb{I} \ar[d, " \tld{\pi} = proj "]\\
& \bb{A}_1 \x \bb{I} \ar[d]\\
(x - \eta(x,n), x + \eta(x,n)) \ar[ur, " \mLa{\overline{\ga}} "] \ar[uur, " \mLa{0 \x \overline{\ga}} "] \ar[r, " \mLa{\ga} "] & \lqt{\bb{A}_1}{\Ga_x} \x \bb{I} \cong U_x
\ecd
\eew
Projecting \(0 \x \overline{\ga}\) via \(\bb{A}_2 \x \bb{A}_1 \x \bb{I} \to \lt(\lqt{\bb{A}_2 \x \bb{A}_1}{\Ga_y}\rt) \x \bb{I}\) defines a local lift (\(C^1\)) of \(\ga\) along \(\pi\) over the region \((x - \eta(x,n), x + \eta(x,n))\); denote this lift by \(\tld{\ga}(x,n)\).

Note that on the region \(I_i \cap (x - \eta(x,n), x + \eta(x,n))\) (where \(\ga \in B_i\)) the curve \(\tld{\ga}(x,n)\) lies in \(E_j\).  Indeed:
\ew
\tld{\ga}(x,n)|_{I_i \cap (x - \eta(x,n), x + \eta(x,n))} \pc [0] \x \bb{I} \cc \lt(\lqt{\bb{A}_2 \x \bb{A}_1}{\Ga_y}\rt) \x \bb{I}
\eew
and one may verify that \([0] \x \bb{I}\) lies in the stratum \(E_j\).  Hence by \cite[Lem.\ 3.1.7, p.\ 217]{GMT}, \(g_j\lt(\dot{\tld{\ga}}(x,n)_t\rt)\) is well-defined for almost every \(t \in I_i \cap (x - \eta(x,n), x + \eta(x,n))\) and thus for almost every \(t \in \tld{I}_i \cap (x - \eta(x,n), x + \eta(x,n))\).  Now consider the function:
\ew
g_j\lt(\dot{\tld{\ga}}(x,n)\rt)^\frac{1}{2} - \mc{L}_i\lt(\dot{\ga}\rt)^\frac{1}{2} \text{ where defined on } \tld{I}_i \cap (x - \eta(x,n), x + \eta(x,n)).
\eew
This is a continuous, non-negative map which is less than \(\frac{1}{n}\) at the point \(x\).  Therefore, by reducing \(\eta(x,n) > 0\) if necessary, one may ensure that:
\e\label{close-to-opt}
g_j\lt(\dot{\tld{\ga}}(x,n)\rt)^\frac{1}{2} - \mc{L}_i\lt(\dot{\ga}\rt)^\frac{1}{2} < \frac{2}{n} \text{ almost everywhere on } \tld{I}_i \cap (x - \eta(x,n), x + \eta(x,n)).
\ee

Now consider the collection of subsets of \(I\) given by:
\ew
\mc{S}_n = \lt\{ (x - r, x + r) ~\m|~ x \in \tld{I}, r \in (0,\eta(x,n)) \rt\}.
\eew
By applying the Vitali Covering Theorem \cite[Thm.\ 2.2, p.\ 26]{GoS&MiES:F&R}, there exist \(x_p \in \tld{I}\) and \(r_p \in (0,\eta(x_p,n))\), for \(p \in \bb{N}\), such that:
\begin{itemize}
\item The sets \(\lt\{(x_p - r_p, x_p + r_p)\rt\}_{p \in \bb{N}}\) are disjoint;
\item \(\cal{L}\lt[ I \m\osr \bigcup_{p \in \bb{N}} (x_p - r_p, x_p + r_p) \rt] = \cal{L}\lt[ \tld{I} \m\osr \bigcup_{p \in \bb{N}} (x_p - r_p, x_p + r_p) \rt] = 0\).
\end{itemize}
Choose \(N = N(n)\) sufficiently large such that:
\e\label{almost-cover}
\cal{L}\lt[ I \m\osr \bigcup_{p=0}^{N(n)} (x_p - r_p, x_p + r_p) \rt] < \frac{1}{n}.
\ee

Now construct the lift \(\tld{\ga}_n\) as follows:
\begin{itemize}
\item On each set \(\lt(x_p - \frac{n-1}{n}r_p, x_p + \frac{n-1}{n}r_p\rt)\), \(p = 0,...,N(n)\), define:
\ew
\tld{\ga}_n = \tld{\ga}(x_p,n)|_{\lt(x_p - \frac{n-1}{n}r_p, x_p + \frac{n-1}{n}r_p\rt)}.
\eew

\item Since the open sets \(\lt\{(x_p - r_p, x_p + r_p)\rt\}_{p \in \{0,...,N(n)\}}\) are disjoint, the complement of the union of the smaller open sets \(\lt\{\lt(x_p - \frac{n-1}{n}r_p, x_p + \frac{n-1}{n}r_p\rt)\rt\}_{p \in \{0,...,N(n)\}}\) is a finite collection of closed intervals, including two intervals of the form \([0,\al]\) and \([\be,1]\).  On each of these closed intervals, use \lref{lifting-lem} to choose some piecewise-\(C^1\) lift of \(\ga\) along \(\pi\), with endpoints chosen so that the resulting lift \(\tld{\ga}_n\) is piecewise-\(C^1\) and so that \(\tld{\ga}_n(0)\), \(\tld{\ga}_n(1)\) take the required values in \(\pi^{-1}(\ga(0))\) and \(\pi^{-1}(\ga(1))\) respectively.
\end{itemize}

I now claim that:
\e\label{conv-in-meas}
\cal{L}\lt[ \lt\{ x \in I ~\m|~ \h{g}\lt(\dot{\tld{\ga}}_n\rt)^\frac{1}{2}|_x - \h{\mc{L}}\lt(\dot{\ga}\rt)^\frac{1}{2}|_x \ge \frac{2}{n} \rt\} \rt]< \frac{3}{n},
\ee
a result which would imply the convergence of the functions \(\h{g}\lt(\dot{\tld{\ga}}_n\rt)^\frac{1}{2} \to \h{\mc{L}}\lt(\dot{\ga}\rt)^\frac{1}{2}\) in measure.  To verify \eref{conv-in-meas}, for each \(x_p\) choose \(i(p)\) such that \(x_p \in \tld{I}_{i(p)}\) and recall from \eref{close-to-opt} that:
\ew
\h{g}\lt(\dot{\tld{\ga}}_n\rt)^\frac{1}{2} - \h{\mc{L}}\lt(\dot{\ga}\rt)^\frac{1}{2} < \frac{2}{n} \text{ almost everywhere on } \tld{I}_{i(p)} \cap \lt(x_p - \frac{n-1}{n}r_p, x_p + \frac{n-1}{n}r_p\rt).
\eew

Therefore:
\ew
\cal{L}\lt[ \lt\{ x \in I ~\m|~ \h{g}\lt(\dot{\tld{\ga}}_n\rt)^\frac{1}{2}|_x - \h{\mc{L}}\lt(\dot{\ga}\rt)^\frac{1}{2}|_x \ge \frac{2}{n} \rt\} \rt] &\le 1 - \cal{L} \lt[ \lt. \bigcup_{p=0}^{N(n)} \lt\{ \tld{I}_{i(p)} \cap \lt(x_p - \frac{n-1}{n}r_p, x_p + \frac{n-1}{n}r_p\rt) \rt\} \rt. \rt]\\
&= 1 - \sum_{p=0}^{N(n)} \cal{L} \lt[ \tld{I}_{i(p)} \cap \lt(x_p - \frac{n-1}{n}r_p, x_p + \frac{n-1}{n}r_p\rt) \rt],
\eew
where the final equality follows from the fact that the union is disjoint.  Now from \eref{density-close-to-1}, for each \(p\):
\ew
\cal{L} \lt[ \tld{I}_{i(p)} \cap \lt(x_p - \frac{n-1}{n}r_p, x_p + \frac{n-1}{n}r_p\rt) \rt] &\ge \frac{n-1}{n}\cal{L}\lt[ \lt(x_p - \frac{n-1}{n}r_p, x_p + \frac{n-1}{n}r_p\rt) \rt]\\
&= \lt(\frac{n-1}{n}\rt)^2\cal{L}\lt[ \lt(x_p - r_p, x_p + r_p\rt) \rt],
\eew
and therefore:
\ew
\sum_{p=0}^{N(n)} \cal{L} \lt[ \tld{I}_{i(p)} \cap \lt(x_p - \frac{n-1}{n}r_p, x_p + \frac{n-1}{n}r_p\rt) \rt] &\ge \lt(\frac{n-1}{n}\rt)^2 \sum_{p=0}^{N(n)} \cal{L} \lt[ \lt(x_p - r_p, x_p + r_p\rt) \rt]\\
&= \lt(\frac{n-1}{n}\rt)^2 \cal{L} \lt[ \bigcup_{p=0}^{N(n)} \lt(x_p - r_p, x_p + r_p\rt) \rt].
\eew
By \eref{almost-cover}:
\ew
\cal{L} \lt[ \bigcup_{p=0}^{N(n)} \lt(x_p - r_p, x_p + r_p\rt) \rt] > \frac{n-1}{n},
\eew
hence:
\ew
\sum_{p=0}^{N(n)} \cal{L} \lt[ \tld{I}_{i(p)} \cap \lt(x_p - \frac{n-1}{n}r_p, x_p + \frac{n-1}{n}r_p\rt) \rt] > \lt(\frac{n-1}{n}\rt)^3
\eew
and whence:
\ew
\cal{L}\lt[ \lt\{ x \in I ~\m|~ \h{g}\lt(\dot{\tld{\ga}}_n\rt)^\frac{1}{2}|_x - \h{\mc{L}}\lt(\dot{\ga}\rt)^\frac{1}{2}|_x \ge \frac{2}{n} \rt\} \rt] < 1 - \lt(\frac{n-1}{n}\rt)^3 \le \frac{3}{n}.
\eew
This completes the proof of \lref{DCT-seq-constr}.

\end{proof}

Using \lrefs{DCT-bd}, \ref{lifting-lem} and \ref{DCT-seq-constr}, I now prove \pref{FMS-Thm}:

\begin{proof}[Proof of \pref{FMS-Thm}]
By the definition of \(\h{\mc{L}}\), for all \(e,e' \in E\) and all piecewise-\(C^1\) paths \(\ga:e \to e'\):
\ew
\h{g}\lt(\dot{\ga}\rt) \ge \h{\mc{L}}\lt(\overset{\mb{\1}}{\pi(\ga)}\rt) ~ \text{almost everywhere},
\eew
and hence:
\ew
\ell_{\h{g}}(\ga) \ge \ell_{\h{\mc{L}}}(\pi(\ga)) \ge d^{\h{\mc{L}}}(\pi(e),\pi(e')).
\eew
Taking the infimum over all such \(\ga\) shows that:
\ew
d^{\h{g}}(e,e') \ge d^{\h{\mc{L}}}(\pi(e),\pi(e')).
\eew

Conversely, let \(\ga:[0,1] \to B\) be a piecewise-\(C^1\) path \(\pi(e) \to \pi(e')\).  By \lref{DCT-seq-constr}, there exists a sequence of piecewise-\(C^1\) lifts \(\tld{\ga}_n: e \to e'\) of \(\ga\) along \(\pi\) such that:
\ew
\h{g}\lt(\dot{\tld{\ga}}_n\rt)^\frac{1}{2} \to \h{\mc{L}}\lt(\dot{\ga}\rt)^\frac{1}{2} ~ \text{in measure as } n \to \0.
\eew
By the {\it a priori} bound in \lref{DCT-bd}, the Dominated Convergence Theorem applies, and so:
\ew
d^{\h{g}}(e,e') \le \lim_{n \to \0} \ell_{\h{g}}\lt(\tld{\ga}_n\rt) = \lim_{n \to \0} \int_{[0,1]} \h{g}\lt(\dot{\tld{\ga}}_n\rt)^\frac{1}{2} \dd\cal{L} = \int_{[0,1]} \h{\mc{L}}\lt(\dot{\ga}\rt) \dd\cal{L}^\frac{1}{2} = \ell_{\h{\mc{L}}}(\ga).
\eew
Taking the infimum over \(\ga\) completes the proof.

\end{proof}

\section{Proof of \tref{Cvgce-Thm-1-FV}: Part 2}

The purpose of this section is to complete the proof of \tref{Cvgce-Thm-1-FV}.  By applying \pref{2stage} and the results of \sref{Gen-Coll}, it suffices to prove the following result:

\begin{Prop}\label{SM-conv-thm}
Let notation be as in \ntref{cvgce-nt} and assume that conditions (i)--(iv) in \tref{Cvgce-Thm-1-FV} hold.  Then:
\ew
\fr{D}\lt[ \lt(E_1, d^\mu\rt) \to \lt(E_2, d^\0\rt) \rt] \to 0 \as \mu \to \0.
\eew
\end{Prop}

The proof proceeds via a series of lemmas.

\subsection{Estimating Lipschitz constants}

\begin{Lem}\label{LC-vs-norm}
Let \((\bb{A},g)\) be a finite-dimensional inner product space and write \(\|-\|_g\) for the norm induced by \(g\) on the space \(\ss{2}\bb{A}^*\) of symmetric bilinear forms.  Then for any \(h \in \ss{2}\bb{A}^*\):
\e\label{lc-vs-norm-eqn}
h \le \|h\|_g\cdot g.
\ee
\end{Lem}

\begin{proof}
Firstly, recall the definition of \(\|h\|_g\). Pick any \(g\)-orthonormal basis \((a_1,...,a_n)\) of \(\bb{A}\). Then:
\ew
\|h\|_g = \sqrt{\sum_{i,j=1}^n h(a_i,a_j)^2}.
\eew
Now fix any vector \(a \in \bb{A}\).  By scale invariance of \eref{lc-vs-norm-eqn}, one may assume \wlg\ that \(g(a) = 1\).  One can then extend \(a\) to a \(g\)-orthonormal basis \((a_1 = a,...,a_n)\) of \(\bb{A}\) and compute:
\ew
h(a) &\le |h(a)|\\
&= \sqrt{h(a)^2}\\
&\le \sqrt{\sum_{i,j=1}^n h(a_i,a_j)^2} = \|h\|_g,
\eew
as required.

\end{proof}

\subsection{Convergence on the regions \(E^{(r)}\)}

Consider the region \(E^{(r)} \pc E_1 \osr S_1 \cong E_2 \osr S_2\).  The restriction of each stratified Riemannian metric \(\h{g}^\mu\) to \(E^{(r)}\) induces a length structure and hence a metric on \(E^{(r)}\), denoted \(d^{\mu,r}\).  (Note that, in general, \(d^{\mu,r} \ne d^\mu|_{E^{(r)}}\) since the metric on the left-hand side is intrinsic, defined by optimising over the length of paths contained only in \(E^{(r)}\), while the metric on the right-hand side is extrinsic, defined by optimising over the length of paths in \(E_1\).)  Analogously, the restriction of the stratified Riemannian semi-metric \(\h{g}^\0\) to \(E^{(r)}\) induces a weak length structure and hence a semi-metric on \(E^{(r)}\) denoted \(d^{\0,r}\) (where again \(d^{\0,r} \ne d^\0|_{E^{(r)}}\) in general).

\begin{Lem}\label{smooth-conv}
Assuming condition (i) from \tref{Cvgce-Thm-1-FV}, for all fixed \(r \in (0,1]\):
\ew
d^{\mu,r} \to d^{\0,r} \text{ uniformly as } \mu \to \0.
\eew
\end{Lem}

\begin{proof}
Fix \(x,y \in E^{(r)}\) and let \(\ga\) be any piecewise-\(C^1\) path from \(x\) to \(y\) in \(E^{(r)}\).  Since \(\h{g}^\mu \to \h{g}^\0\) uniformly on \(E^{(r)}\), it follows that \(\ell_{\h{g}^\mu}(\ga) \to \ell_{\h{g}^\0}(\ga)\) as \(\mu \to \0\).  Moreover, one clearly has \(d^{\mu,r}(x,y) \le \ell_{\h{g}^\mu}(\ga)\) for each \(\mu\).  Taking first the limit superior over \(\mu\) and then the infimum over all \(\ga\) in this inequality, therefore, yields:
\e\label{limsup<}
\limsup_\mu d^{\mu,r}(x,y) \le d^{\0,r}(x,y).
\ee

Conversely, for any \(\ga\) as above, by \eref{La-cvgce} in condition (i), one has:
\ew
\ell_{\h{g}^\mu}(\ga) \ge \La_\mu\ell_{\h{g}^\0}(\ga) \ge \La_\mu d^{\0,r}(x,y).
\eew
Since \(\La_\mu \to 1\) as \(\mu \to \0\), taking firstly the infimum over all \(\ga\) and then the limit inferior over all \(\mu\) yields:
\e\label{liminf>}
\liminf_\mu d^{\mu,r}(x,y) \ge d^{\0,r}(x,y).
\ee
Combining \erefs{limsup<} and \eqref{liminf>} gives:
\ew
d^{\0,r}(x,y) \le \liminf_\mu d^{\mu,r}(x,y) \le \limsup_\mu d^{\mu,r}(x,y) \le d^{\0,r}(x,y)
\eew
for all \(x, y \in E^{(r)}\) and hence \(d^{\mu,r} \to d^{\0,r}\) pointwise on \(E^{(r)}\).

Now fix some reference stratified Riemannian metric \(\h{h}\) on \(E_1\) and write \(d^r\) for the (intrinsic) metric on \(E^{(r)}\) induced by \(\h{h}|_{E^{(r)}}\). By assumption, for each stratum \(E_i\) of \(E_1 \osr S_1\):
\ew
\|g^\mu_i - g^\0_i\|_{h_i} \to 0 \text{ uniformly on } E_i \cap E^{(r)} \text{ as } \mu \to \0,
\eew
where \(\|-\|_{h_i}\) denotes the pointwise norm on symmetric bilinear forms induced by \(h_i\), as in \lref{LC-vs-norm}.  Next, note that since \(g^\0_i\) can be continuously extended to a semi-metric \(\overline{g}^\0_i\) on all of \(E_2\) (and likewise for \(h_i\)), then, by compactness of \(E_2\), there exists some constant \(C_i > 0\) such that \(\|g^\0_i\|_{h_i} \le \frac{C_i}{2}\) on all of \(E_i \cap E^{(r)}\).  It follows that for all sufficiently large \(\mu\):
\ew
\|g^\mu_i\|_{h_i} \le C_i \text{ on } E_i \cap E^{(r)}.
\eew
Thus for \(C \ge \max_i C_i\) sufficiently large, it follows from \lref{LC-vs-norm} that for all \(\mu\) (including \(\mu = \0\)) \(\h{g}^\mu \le C \h{h}\) and in particular:
\ew
d^{\mu, r} \le C^\frac{1}{2} d^r.
\eew
By applying the triangle inequality, it follows that for all pairs \((x,y),(x',y') \in E^{(r)} \x E^{(r)}\) and all \(\mu\) (including \(\mu=\0\)):
\ew
\lt|d^{\mu,r}(x,y) - d^{\mu,r}(x',y')\rt| \le C^\frac{1}{2}\lt(d^r(x,x') + d^r(y,y')\rt)
\eew
and thus the family of functions \(d^{\mu,r}: E^{(r)} \x E^{(r)} \to \bb{R}\) is uniformly Lipschitz and hence equicontinuous.  By combining this equicontinuity with the pointwise convergence \(d^{\mu,r} \to d^{\0,r}\), the proof of \pref{smooth-conv} is now completed by the following variant of the well-known Ascoli--Arzel\`{a} theorem:
\begin{Thm}
Let \((X,d)\) be a compact metric space and let \(f_n:X \to \bb{R}\) be equicontinuous functions converging pointwise to a continuous function \(f\).  Then \(f_n \to f\) uniformly.
\end{Thm}

The proof is simple, therefore, I omit it.  This completes the proof of \lref{smooth-conv}.

\end{proof}

\subsection{Combinatorial preliminaries}

Recall that the `singular' regions \(S_1(j)\) are indexed by \(j \in \{1,...,N\}\) and likewise for \(E_2\).  Recall also the sets \(\del^{(r)}(j) = \del U_1^{(r)}(j) \cong \del U_2^{(r)}(j)\).  Write \([N] = \{1,...,N\}\) and given any \(1 \le k \le N\), let \([N]^{(k)}\) denote the set of ordered tuples of \(k\) distinct elements of \([N]\), which will be denoted \((j_1,...,j_k)\).  For notational convenience, use \(\w\) to denote the binary minimum of two numbers, i.e.: \(a \w b = \min(a,b)\).

\begin{Lem}\label{comb-decomp}
Fix \(r \in (0,1]\) and let \(x,y \in E^{(r)}\).  Then for all \(\mu \ge 1\) (including \(\mu = \0\)):
\e\label{comb-dec-eqn}
d^\mu(x &,y) = d^{\mu,r}(x,y) \w\\
& \min_{1 \le k \le N} \lt[ \min_{\lt(j_1,...,j_k\rt) \in [N]^{(k)}} \lt( \inf_{\substack{x_i,y_i \in \del^{(r)}(j_i) \\ i = 1,...,k}} d^{\mu,r}(x,x_1) + d^{\mu,r}(y_n,y) + \sum_{i=1}^k d^\mu(x_i,y_i) + d^{\mu,r}(y_i,x_i) \rt) \rt].
\ee

In particular:
\ew
\sup_{x,y \in E^{(r)}} \lt|d^\mu(x,y) - d^\0(x,y)\rt| \le &(N+2)\sup_{x',y' \in E^{(r)}} \lt|d^{\mu,r}(x',y') - d^{\0,r}(x',y')\rt|\\
& + N \max_{j \in [N]} \sup_{x'',y'' \in \del^{(r)}(j)} \lt|d^\mu(x'',y'') - d^\0(x'',y'')\rt|.
\eew
\end{Lem}

\begin{proof}
It \stp\ \eref{comb-dec-eqn}, the final claim being a direct consequence of this.  Eqn.\hs{0.6mm}(\ref{comb-dec-eqn}) is proved in the case \(\mu < \0\), the case \(\mu = \0\) being identical.

Write \(\Om(x,y)\) for the set of all piecewise-\(C^1\) paths \(\ga\) from \(x\) to \(y\) in \(E_1\) and \(\Om_0(x,y) \cc \Om(x,y)\) for the set of all piecewise-\(C^1\) paths \(\ga\) from \(x\) to \(y\) which lie entirely within \(E^{(r)}\).  Let \(\ga \in \Om(x,y) \osr \Om_0(x,y)\) and write \([a,b] \pc \bb{R}\) for the domain of \(\ga\).  Assign to \(\ga\) an index \(j_1(\ga) \in [N]\) and a number \(t_1(\ga) \in [a,b]\) as follows:

Define:
\ew
t_0(\ga) = \inf\lt\{t \in [a,b] ~\m|~ \ga(t) \in U^{(r)}_1(j) \text{ for some } j \in [N] \rt\},
\eew
the right-hand side being non-empty, precisely because \(\ga \notin \Om_0(x,y)\).  Then by \eref{dj} in condition 3 of \ntref{cvgce-nt}, there is a unique \(j \in [N]\) such that \(\ga(t_0(\ga)) \in \overline{U^{(r)}_1(j)}\); denote this unique \(j\) by \(j_1(\ga)\).  Now define:
\ew
t_1(\ga) = \sup\lt\{t \in [a,b] ~\m|~ \ga(t) \in U^{(r)}_1(j_1(\ga)) \rt\}.
\eew

Now suppose that \(t_1(\ga) < b\) and that the path \(\ga|_{[t_1(\ga),b]} \notin \Om_0(\ga(t_1(\ga)),b)\).  Then one may define:
\ew
j_2(\ga) = j_1(\ga|_{[t_1(\ga),b]}) \et t_2(\ga) = t_1(\ga|_{[t_1(\ga),b]}).
\eew
(Observe that \(j_2(\ga) \ne j_1(\ga)\), since \(\ga\) never lies in \(U^{(r)}_1(j_1(\ga))\) after time \(t_1(\ga)\).)  One may continue in this vein, defining:
\ew
j_{k+1}(\ga) = j_1(\ga|_{[t_k(\ga),b]}) \et t_{k+1}(\ga) = t_1(\ga|_{[t_k(\ga),b]}),
\eew
until either \(t_k(\ga) = b\) for some \(k\) or \(\ga|_{[t_k(\ga),b]} \in \Om_0(\ga(t_k(\ga)),b)\), one of these two conditions necessarily being reached for some \(k \in [N]\) due to the fact that the \(j_1(\ga), j_2(\ga), j_3(\ga),...\) are all distinct elements of the finite set \([N]\).  Call:
\ew
(j_1(\ga),...,j_k(\ga)) \in [N]^{(k)}
\eew
the characteristic tuple of \(\ga\).

Now for each \(k \in [N]\) and each tuple \({\bf i} \in [N]^{(k)}\), define \(\Om_{\bf i}(x,y)\) to be the set of all \(\ga \in \Om(x,y)\) with characteristic tuple \(\bf i\).  The above discussion shows that there is a disjoint union:
\ew
\Om(x,y) = \Om_0(x,y) \coprod \lt(\coprod_{\substack{k \in [N]\\ {\bf i} \in [N]^{(k)}}} \Om_{\bf i}(x,y)\rt).
\eew
However, by definition:
\ew
\inf_{\ga \in \Om_0(x,y)} \ell_{\h{g}^\mu}(\ga) = d^{\mu,r}(x,y).
\eew
Similarly, for each \(k \in [N]\) and \({\bf i} \in [N]^{(k)}\), one may verify that:
\ew
\inf_{\ga \in \Om_{\bf{i}}(x,y)} \ell_{\h{g}^\mu}(x,y) = \inf_{\substack{x_i,y_i \in \del^{(r)}(j_i)\\ i = 1,...,k}} d^{\mu,r}(x,x_1) + d^{\mu,r}(y_n,y) + \sum_{i=1}^k d^\mu(x_i,y_i) + d^{\mu,r}(y_i,x_i).
\eew
The result now follows.

\end{proof}

\subsection{Completing the proof of \tref{SM-conv-thm}}

Let \(\tld{\Ph}\) be some fixed, possibly discontinuous, extension of the map \(\Ph: E_1 \osr S_1 \to E_2 \osr S_2\) such that \(\tld{\Ph}(S_1) \cc S_2\) and the following diagram:
\ew
\bcd
U^{(1)}_1(j) \ar[rr, " \tld{\Ph} "] \ar[rd, " \fr{f}_{1,j} "'] & & U^{(1)}_2(j) \ar[dl, " \fr{f}_{2,j} "]\\
& S(j) &
\ecd
\eew
commutes for each \(j \in \{1,...,N\}\).  Explicitly, for each \(j \in \{1,...,N\}\) and each \(x \in S_1(j)\), choose some point \(y \in \fr{f}_{2,j}^{-1}(\{\fr{f}_{1,j}(x)\}) \cap S_2(j)\) and define \(\tld{\Ph}(x) = y\).  Recalling the definition of forwards discrepancy from \sref{H&GHD}, to prove \tref{SM-conv-thm} it suffices to prove that, given any \(\eta>0\), \(\tld{\Ph}\) is an \(\eta\)-isometry \((E_1,d^\mu) \to (E_2,d^\0)\) whenever \(\mu\) is sufficiently large (depending on \(\eta\)).  This is equivalent to the following two lemmas:

\begin{Lem}\label{e-net}
\(\tld{\Ph}(E_1) \cc E_2\) is dense \wrt\ the semi-metric \(d^\0\).
\end{Lem}

\begin{Lem}\label{fd->0}
\ew
\lim_{\mu \to \0} \sup_{x,y \in E_1} \lt|d^\mu(x,y) - d^\0\lt(\tld{\Ph}(x),\tld{\Ph}(y)\rt)\rt| = 0.
\eew
\end{Lem}

\begin{proof}[Proof of \lref{e-net}]
Clearly \(\tld{\Ph}(E_1) \supseteq \Ph(E_1 \osr S_1) = E_2 \osr S_2\).  Thus, to prove that \(\tld{\Ph}(E_1)\) is dense, it suffices to prove that for all \(j \in \{1,...,N\}\), all \(w \in S_2(j)\) and all \(\eta>0\), there exists \(x \in E_1\) such that:
\ew
d^\0\lt(w, \tld{\Ph}(x)\rt) \le \eta.
\eew
To this end, given \(w \in S_2(j)\), choose \(x\) to be any point of \(\fr{f}_{1,j}^{-1}(\{\fr{f}_{2,j}(w)\}) \cap S_1(j)\).  Then by definition of \(\tld{\Ph}\):
\ew
\tld{\Ph}(x) &\in \fr{f}_{2,j}^{-1}(\{\fr{f}_{2,j}(w)\}) \cap S_2(j)\\
&\cc \fr{f}_{2,j}^{-1}(\{\fr{f}_{2,j}(w)\}) \cap U^{(r)}_2(j) \text{ for any } r \in (0,1].
\eew
Hence for all \(r \in (0,1]\):
\ew
d^\0\lt(w, \tld{\Ph}(x)\rt) &\le \diam_{d^\0}\lt[ \fr{f}_{2,j}^{-1}(\{\fr{f}_{2,j}(w)\}) \cap U^{(r)}_2(j) \rt]\\
&\le \max_{j \in \{1,...,N\}} ~ \sup_{p \in S(j)} ~ \diam_{d^\0}\lt[ \fr{f}_{2,j}^{-1}(p) \cap U^{(r)}_2(j) \rt]
\eew
and thus:
\ew
d^\0\lt(w, \tld{\Ph}(x)\rt) \le \limsup_{r \to 0} ~ \max_{j \in \{1,...,N\}} ~ \sup_{p \in S(j)} ~ \diam_{d^\0}\lt[ \fr{f}_{2,j}^{-1}(p) \cap U^{(r)}_2(j) \rt] = 0,
\eew
the final equality being condition (iii) in \tref{Cvgce-Thm-1-FV}.  The result follows.

\end{proof}

\begin{proof}[Proof of \lref{fd->0}]
Pick \(x,y \in E_1\) and let \(r \in (0,1]\).  Define (potentially) new points \(x' = x'(r)\) and \(y' = y'(r)\) in \(E^{(r)}\) via:
\ew
x' =
\begin{dcases*}
x & if \(x \in E^{(r)}\);\\
\text{some point in } E^{(r)} \cap \overline{\fr{f}_{1,j}^{-1}(\{\fr{f}_{1,j}(x)\}) \cap U^{(r)}_1(j)} \pc \del^{(r)}(j) & if \(x \in U^{(r)}_1(j)\)
\end{dcases*}
\eew
and analogously for \(y'\).  Note that:
\e\label{shift-bd1}
d^\mu(x,x'), d^\mu(y,y') \le \max_{j \in [N]} \sup_{p \in S(j)} \diam_{d^\mu}\lt[ \fr{f}_{1,j}^{-1}(\{p\}) \cap U^{(r)}_1(j) \rt].
\ee

Next consider \(d^\0\lt(\tld{\Ph}(x),x'\rt)\) and \(d^\0\lt(\tld{\Ph}(y),y'\rt)\), where \(x'\) and \(y'\) are identified with \(\tld{\Ph}(x') = \Ph(x')\) and \(\tld{\Ph}(y') = \Ph(y')\) in the usual way.  Clearly if \(x \in E^{(r)}\), then \(\tld{\Ph}(x) = x = x'\) and \(d^\0\lt(\tld{\Ph}(x),x'\rt) = 0\).  Thus suppose \(x \in U^{(r)}_1(j)\) for some \(j\).  The commutative diagram:
\ew
\bcd
\lt. U^{(1)}_1(j) \m\osr S_1(j) \rt. \ar[rr, " \Ph "] \ar[rd, " \fr{f}_{1,j} "'] & & \lt. U^{(1)}_2(j) \m\osr S_2(j) \rt. \ar[dl, " \fr{f}_{2,j} "]\\
& S(j) &
\ecd
\eew
shows that:
\ew
x' \in E^{(r)} \cap \overline{\fr{f}_{1,j}^{-1}(\{\fr{f}_{1,j}(x)\}) \cap U^{(r)}_1(j)} \cong E^{(r)} \cap \overline{\fr{f}_{2,j}^{-1}(\{\fr{f}_{1,j}(x)\}) \cap U^{(r)}_2(j)} \pc \del^{(r)}(j),
\eew
using the identification \(\Ph\) in the usual way.  Moreover, from the definition of \(\tld{\Ph}\) one may verify that \(\tld{\Ph}(x) \in \fr{f}_{2,j}^{-1}(\{\fr{f}_{1,j}(x)\}) \cap U^{(r)}_2(j)\).  Thus:
\e\label{shift-bd2}
d^\0\lt(\tld{\Ph}(x), x'\rt) \le \max_{j \in [N]} \sup_{p \in S(j)} \diam_{d^\0}\lt[ \fr{f}_{2,j}^{-1}(\{p\}) \cap U^{(r)}_2(j) \rt].
\ee
The same bound holds for \(d^\0\lt(\tld{\Ph}(y), y'\rt)\).  Thus by the triangle inequality:
\ew
\lt|d^\mu(x,y) - d^\0\lt(\tld{\Ph}(x),\tld{\Ph}(y)\rt)\rt| &\le \lt|d^\mu(x',y') - d^\0(x',y')\rt|\\
& \hs{15mm} + d^\mu(x,x') + d^\mu(y,y')\\
& \hs{15mm} + d^\0\lt(\tld{\Ph}(x),x'\rt) + d^\0\lt(y',\tld{\Ph}(y)\rt)\\
&\le \lt|d^\mu(x',y') - d^\0(x',y')\rt|\\
& \hs{15mm} + 2\max_{j \in [N]} \sup_{p \in S(j)} \diam_{d^\mu}\lt[ \fr{f}_{1,j}^{-1}(\{p\}) \cap U^{(r)}_1(j) \rt]\\
& \hs{15mm} + 2\max_{j \in [N]} \sup_{p \in S(j)} \diam_{d^\0}\lt[ \fr{f}_{2,j}^{-1}(\{p\}) \cap U^{(r)}_2(j) \rt],
\eew
where \erefs{shift-bd1} and \eqref{shift-bd2} have been used in passing to the final inequality.  Taking supremum over \(x\) and \(y\) and applying \lref{comb-decomp} yields:
\ew
\sup_{x, y \in E_1} \lt|d^\mu(x,y) - d^\0\lt(\tld{\Ph}(x),\tld{\Ph}(y)\rt)\rt| &\le (N+2)\sup_{x',y' \in E^{(r)}} \lt|d^{\mu,r}(x',y') - d^{\0,r}(x',y')\rt|\\
& \hs{15mm} + N \max_{j \in [N]} \sup_{x'',y'' \in \del^{(r)}(j)} \lt|d^\mu(x'',y'') - d^\0(x'',y'')\rt|\\
& \hs{15mm} + 2\max_{j \in [N]} \sup_{p \in S(j)} \diam_{d^\mu}\lt[ \fr{f}_{1,j}^{-1}(\{p\}) \cap U^{(r)}_1(j) \rt]\\
& \hs{15mm} + 2\max_{j \in [N]} \sup_{p \in S(j)} \diam_{d^\0}\lt[ \fr{f}_{2,j}^{-1}(\{p\}) \cap U^{(r)}_2(j) \rt].
\eew

By \pref{smooth-conv}, taking the limit superior as \(\mu \to \0\) yields:
\ew
\limsup_{\mu \to \0}\sup_{x, y \in E_1} \lt|d^\mu(x,y) - d^\0\lt(\tld{\Ph}(x),\tld{\Ph}(y)\rt)\rt| &\le N \limsup_{\mu \to \0} \max_{j \in [N]} \sup_{x'',y'' \in \del^{(r)}(j)} \lt|d^\mu(x'',y'') - d^\0(x'',y'')\rt|\\
& \hs{15mm} + 2 \limsup_{\mu \to \0} \max_{j \in [N]} \sup_{p \in S(j)} \diam_{d^\mu}\lt[ \fr{f}_{1,j}^{-1}(\{p\}) \cap U^{(r)}_1(j) \rt]\\
& \hs{15mm} + 2\max_{j \in [N]} \sup_{p \in S(j)} \diam_{d^\0}\lt[ \fr{f}_{2,j}^{-1}(\{p\}) \cap U^{(r)}_2(j) \rt].
\eew
Moreover, by conditions (ii), (iii) and (iv) in \tref{Cvgce-Thm-1-FV}, taking the limit as \(r \to 0\) yields:
\ew
\limsup_{\mu \to \0}\sup_{x, y \in E_1} \lt|d^\mu(x,y) - d^\0\lt(\tld{\Ph}(x),\tld{\Ph}(y)\rt)\rt| = 0,
\eew
and hence the limit at \(\mu \to \0\) also equals zero, as required.

\end{proof}

Combining \lrefs{e-net} and \ref{fd->0} shows that \(\fr{D}\lt[ \lt(E_1, d^\mu\rt) \to \lt(E_2, d^\0\rt) \rt] \to 0\) as \(\mu \to \0\), completing the proof of \pref{SM-conv-thm}.  The proof of \tref{Cvgce-Thm-1-FV} is now completed by combining \pref{FMS-Thm} and \pref{SM-conv-thm}, together with \tref{2stage}.

\qed

~\vs{5mm}

\noindent Laurence H.\ Mayther\\
University of Cambridge\\
United Kingdom\\
{\it lhm32@cam.ac.uk}

\end{document}